\documentclass[a4paper,10pt]{article}

\usepackage{amsmath,amssymb,amsthm}
\usepackage{amssymb,latexsym, bbm,comment}

\usepackage{epsfig}

\numberwithin{equation}{section}

\newcommand{\origsetminus}{} \let\origsetminus=\setminus           
\renewcommand{\setminus}{\!\origsetminus\!}

\let\oldmarginpar\marginpar
\renewcommand\marginpar[1]{\-\oldmarginpar[\raggedleft\footnotesize #1]%
{\raggedright\footnotesize #1}}

{\theoremstyle{plain}
\newtheorem{lemma}{Lemma}[section]

\newtheorem{theorem}[lemma]{Theorem}
\newtheorem{corollary}[lemma]{Corollary}

} {\theoremstyle{definition}
\newtheorem{example}[lemma]{Example}
\newtheorem{remark}[lemma]{Remark}
}

\renewcommand{\mathbb}{\mathbbm}                     
\renewcommand{\epsilon}{\varepsilon}                 
\renewcommand{\phi}{\varphi}
\renewcommand{\theta}{\vartheta}
\renewcommand{\le}{\leqslant}
\renewcommand{\ge}{\geqslant}

\newcommand{\origfoo}{} \let\origfoo=\sqrt           
\renewcommand{\sqrt}[1]{\origfoo{#1}\;}

\newcommand{\abs}[1]{\left\lvert #1 \right\rvert}    
\newcommand{\norm}[1]{\left\lVert #1 \right\rVert}   
\DeclareMathOperator{\F}{{\cal F}}                   
\DeclareMathOperator{\R}{{\mathbb R}}                

\DeclareMathOperator{\Rp}{{\mathbb R}_+}             
\DeclareMathOperator{\C}{{\mathbb C}}                
\DeclareMathOperator{\Q}{{\mathbb Q}}
\DeclareMathOperator{\N}{{\mathbb N}}                
\DeclareMathOperator{\Id}{ Id}                        

\newcommand{\A}{{\mathcal A}}

\DeclareMathOperator{\Borel}{{\mathcal B}}
\newcommand{\scapro}[2]{\langle #1,#2\rangle}       
\DeclareMathOperator{\1}{\mathbbm 1}

\newcommand{\D}{{\mathcal D}}
\renewcommand{\L}{{\mathcal L}}

\newcommand{\seg}[2]{[\![ #1,#2]\!]}       

\newcounter{zahl}

\renewcommand{\H}{{\mathcal H}}

\title{Non-standard Skorokhod convergence of L{\'e}vy-driven convolution integrals in Hilbert spaces}

\author{
Ilya Pavlyukevich\\
Institut f{\"u}r Mathematik\\
Friedrich--Schiller--Universit{\"a}t Jena\\
Ernst--Abbe--Platz 2\\
07743 Jena\\
Germany\\
\texttt{ilya.pavlyukevich@uni-jena.de}
\and
 Markus Riedle\\
Department of Mathematics\\
King's College\\
Strand\\
London WC2R 2LS\\
United Kingdom\\
\texttt{markus.riedle@kcl.ac.uk}
}

\begin{document}

\maketitle

\begin{abstract}
We study the convergence in probability in the non-standard $M_1$ Skorokhod topology of the Hilbert valued
stochastic convolution integrals of the type $\int_0^t F_\gamma(t-s)\,d L(s)$ to a process $\int_0^t F(t-s)\, d L(s)$ driven by a
L\'evy process $L$. In Banach spaces we introduce strong, weak and product modes of $M_1$-convergence,
prove a criterion for the $M_1$-convergence in probability of stochastically continuous c\`adl\`ag processes
in terms of the convergence in probability of the finite
dimensional marginals and a good behaviour of the corresponding oscillation functions, and establish criteria for the
convergence in probability of L\'evy driven stochastic convolutions. The theory is applied to the infinitely dimensional
integrated Ornstein--Uhlenbeck processes with diagonalisable generators.
\end{abstract}

\noindent
\begin{small}
\textbf{AMS (2000) subject classification:} 60B12$^\ast$, 60F17, 60G51, 60H05.

\end{small}

\smallskip

\noindent
\begin{small}

\end{small}


\section{Introduction}

In many problems of engineering, physics or finance, the evolution of a random system
can be described by stochastic convolution integrals of some kernel with respect to a noise process, see e.g.\ 
Barndorff--Nielsen and Shephard \cite{BarShe03},
Elishakoff \cite{Elishakoff-99}, Pavlyukevich and Sokolov \cite{PavlSok-08}.

The present work is originally motivated by the paper by Chechkin et al.\ \cite{ChechkinGS02},
where the authors consider a simple model for the motion of a charged particle in a constant external magnetic field
subject to $\alpha$-stable L\'evy perturbation. The particle's position $x\in\mathbb R^3$ is described
by the second-order Newtonian equation
\begin{equation*}
 \ddot x=\dot x \times B-\nu \dot x+\varepsilon \dot \ell,
\end{equation*}
where $B\in\mathbb R^3$ is the direction of the magnetic field, $\nu$, $\varepsilon >0$ and
$\ell$ is an isometric three-dimensional $\alpha$-stable L\'evy process with the characteristic function
$E e^{i\langle u, \ell(t)\rangle}=e^{-t |u |^\alpha}$ for all $u\in\mathbb R^3$. 
Denoting the velocity $\dot x=v$ and the linear operator $Av:=-v\times B +\nu v$, we 
obtain that the velocity process $v$ satisfies the linear Ornstein--Uhlenbeck equation
\begin{equation}
\label{eq:ou}
\dot v=-Av+\varepsilon  \dot \ell,
\end{equation}
whereas the coordinate is obtained by integration of the velocity $v$. Assuming that $v_0=x_0=0$ we
solve equation \eqref{eq:ou} explicitly to obtain
$v(t)=\varepsilon \int_0^t e^{-A(t-s)}\, d \ell(s)$, and  Fubini's theorem yields 
$x(t)=\varepsilon A^{-1}\int_0^t(1-e^{-A(t-s)})\,d \ell(s)$.

It is possible to study the dynamics of $x$ in the regime of the small noise perturbation by letting $\varepsilon\to 0$. 
Indeed, performing a convenient time-change
$t\mapsto \varepsilon^{-\alpha}t$, using the self-similarity of $\alpha$-stable processes, i.e.\
\begin{align*}
  \big(\varepsilon \ell(t/\varepsilon^\alpha)\colon  t\ge 0\big)
  \stackrel{\D}{=}\big(\ell(t)\colon t\ge 0\big),
\end{align*}
 and taking
for convenience another copy $L=\ell$ of the driving process
$\ell$, we
transfer the small noise amplitude into the large friction coefficient; that is the stochastic 
processes $X$ and $V$, defined by $ X(t):=x(t/\varepsilon^\alpha)$ and
$V(t):=v(t/\varepsilon^\alpha)$ for all $t\ge 0$, satisfy the equations
\begin{align*}
\dot V=-\frac{1}{\varepsilon^\alpha}AV+\dot L,\qquad\qquad
\dot X=\frac{1}{\varepsilon^\alpha} V.
\end{align*}
By denoting the large parameter $\gamma:=\varepsilon^{-\alpha}$ we obtain the solutions
\begin{align*}
V_\gamma(t)=\int_0^t e^{-\gamma A(t-s)}\, d L(s),\qquad\qquad
AX_\gamma(t)= \int_0^t(1-e^{-\gamma A(t-s)})\,d L(s).
\end{align*}
It can be shown (see Lemma~\ref{le.fddRn}) that if the eigenvalues of $A$ have strictly positive real parts, 
then $A X_\gamma\to L$ in probability with respect to an appropriate metric ($M_1$) in the sample path space. 

\begin{figure}
\begin{center}
\includegraphics[width=3.3cm]{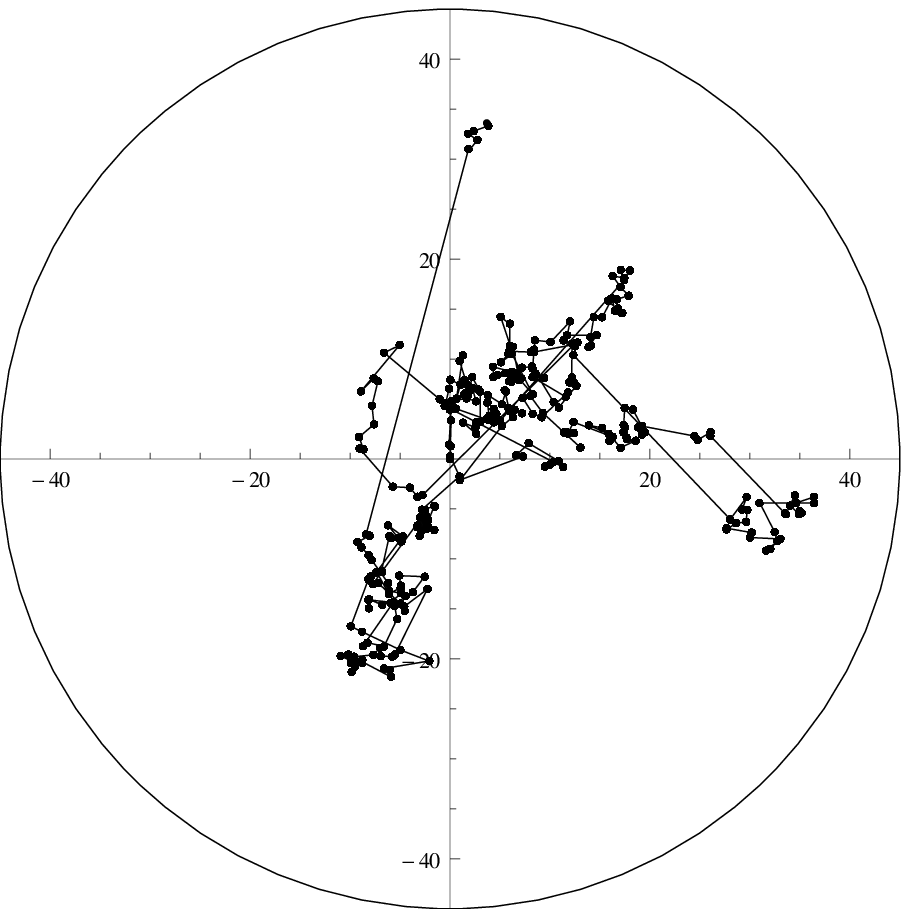}\hfill \includegraphics[width=3.3cm]{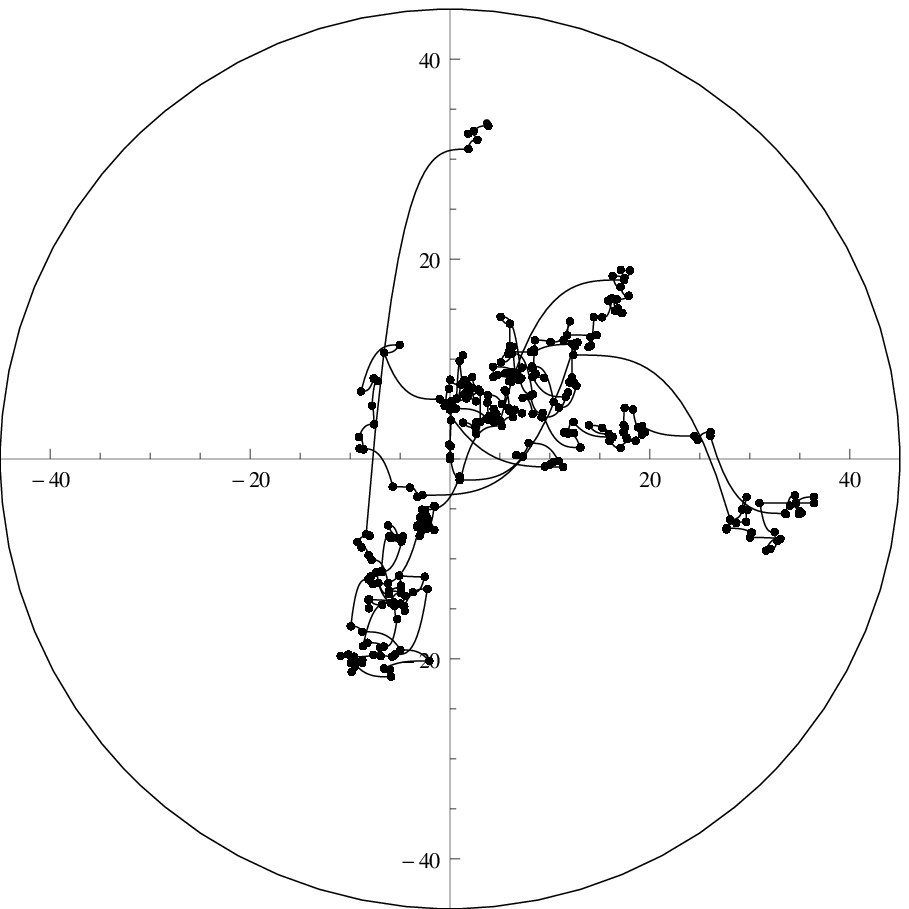}\hfill
\includegraphics[width=3.3cm]{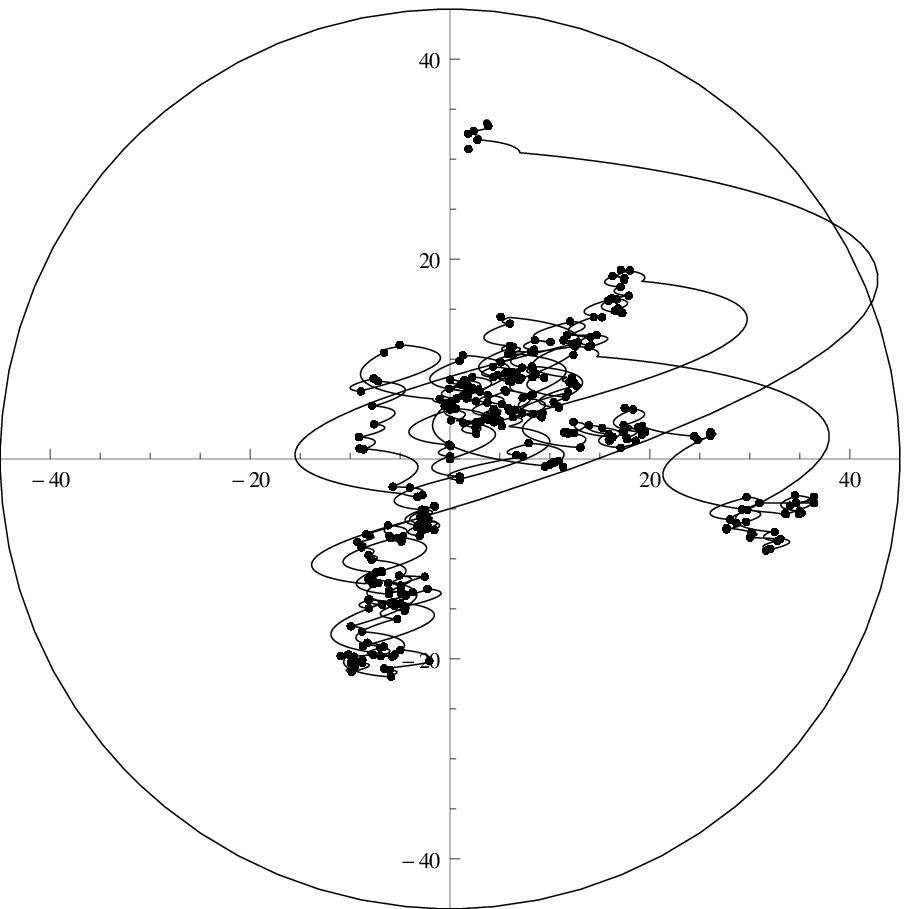}\hfill \includegraphics[width=3.3cm]{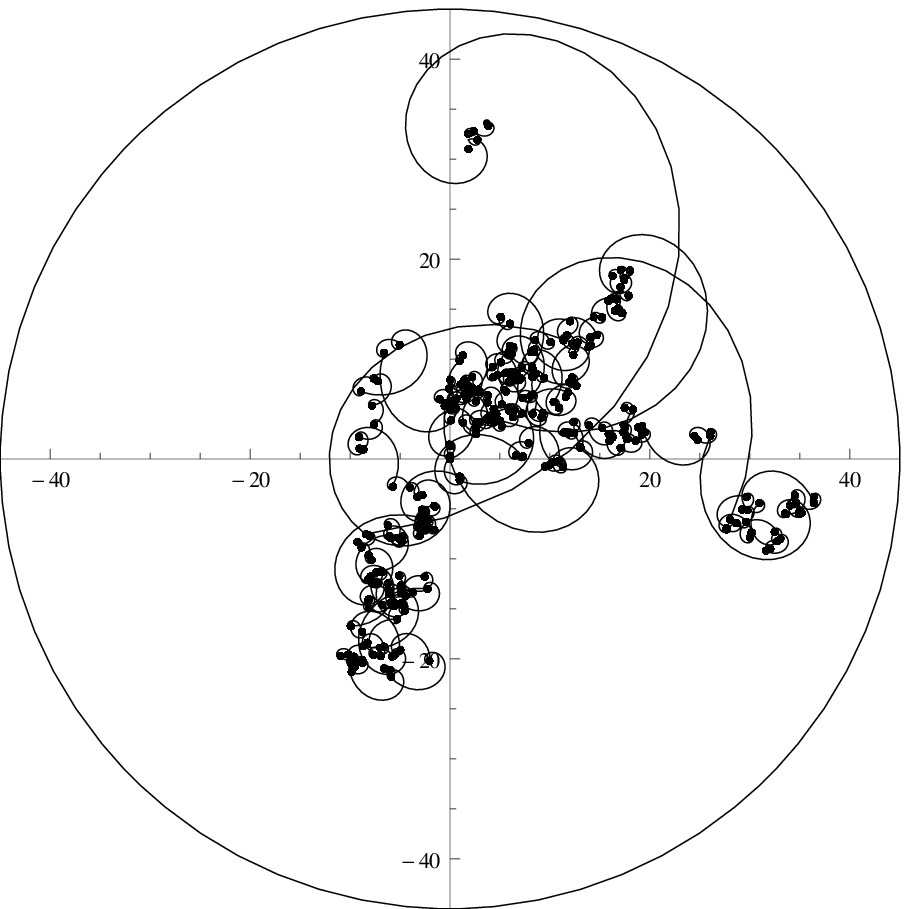}
\end{center}
\caption{Sample paths of the convolution integrals
$A_j X_\gamma^{(j)}(t)=\int_0^t (1-e^{-\gamma A_j(t-s)})\, dL(s)$, $j=1,\dots, 4$, driven by a $1.5$-stable L\'evy process $L$
for large $\gamma>0$ (from left to right).
\label{fig:AX}}
\end{figure}

As an example, consider a two dimensional integrated Ornstein--Uhlenbeck process driven by an $\alpha$-stable L\'evy process $L$
as well as the corresponding sample paths $t\mapsto A_j X_\gamma^{(j)}(t)$ of the integrated Ornstein--Uhlenbeck processes for the following
matrices $A_j$ (see Figure \ref{fig:AX}):
\begin{align*}
A_1=\begin{pmatrix}
   1&0\\
   0&1
  \end{pmatrix},\quad
A_2=\begin{pmatrix}
   1&0\\
   0&3
  \end{pmatrix},\quad
  A_3=\begin{pmatrix}
   1&1\\
   0&1
  \end{pmatrix},\quad
  A_4=\begin{pmatrix}
   1&1 \\
   -1&1
  \end{pmatrix}.
\end{align*}
Obviously, the sample paths differ significantly.
This example determines the scope of this paper: we will establish convergence of general stochastic 
convolution integrals driven by L\'evy processes in the Skorokhod $M_1$ topology in infinite dimensional spaces. 
The particular example \eqref{eq:ou} of this introduction in one dimension is considered by Hintze and Pavlyukevich in \cite{HinPav14}.

As one of four topologies, the $M_1$ topology in the path  space $D([0,1],\R)$, the space of 
c\`adl\`ag functions $f\colon [0,T]\to \R$,  was introduced in the seminal paper 
by Skorokhod \cite{Skorokhod56}. An excellent account on convergence in  the $M_1$ topology in a 
multi-dimensional setting  can be found in Whitt \cite{whitt02}. To the best of our knowledge, 
the $M_1$ topology has not yet been considered in an infinite dimensional setting. Note, that  in the 
$M_1$ topology it is possible that a continuous function, as the sample paths of $AX$, converges 
to a discontinuous function, as the sample paths of the L\'evy process $L$.

In the present paper we study the following aspects of the $M_1$ topology. First, we notice that 
in a typical setting, such as considered in Whitt \cite{whitt02}, one often obtains convergence in the $M_1$ 
topology not only in the weak sense but also in probability. Second, we generalise the 
finite-dimensional setting of Skorokhod and Whitt to stochastic processes with values in separable Banach spaces. Here 
it turns out, that in addition to the two kinds of $M_1$ topologies
in multi-dimensional spaces,  a third kind of $M_1$ topology arises in infinite dimensional spaces.

The second part of our work is concerned with convergence of stochastic convolution integrals in Hilbert 
spaces in the $M_1$ topology. By considering stochastic convolution integrals we may abandon the 
semimartingale setting. It is known, see Basse and Pedersen \cite{BasPed-09} and Basse--O'Connor and 
Rosi\'nski \cite{BasseRosinski}, that even one-dimensional convolution
integrals $\int_0^t F(t-s)\, d L(s)$ define a semimartingale if and only if $F$ is absolutely continuous with sufficiently regular density.

As a specific example, the case of integrated Ornstein--Uhlenbeck  processes in a Hilbert space
is considered in the last section of this paper. It turns out that only in the case of a diagonalisable 
operator, convergence can be established, and then, only in the weakest sense. This result corresponds
to the two-dimensional example above, where only in cases $j=1$ and 
$j=2$ the stochastic convolution integrals converge in the $M_1$ topology. 

\bigskip

\noindent
{\bf Notation:}
For two values $a,b\in\R$ we denote $a\wedge b:=\min\{a,b\}$ and $a\vee b:=\max\{a,b\}$.
The Euclidean norm in $\R^d$, $d\geq 1$, is denoted by $|\cdot |$. A partition $(t_i)_{i=1}^m$
of an interval $[0,T]$ is a finite sequence of numbers $t_i\in [0,T]$
satisfying $t_1<\dots <t_m$.
For functions $f\colon [0,T]\to S$, where $S$ is a linear space with a norm $\|\cdot\|_S$,
we define the supremum norm $\norm{f}_\infty:=\sup_{t\in [0,T]}\norm{f(t)}_{S}$. The 2-variation of a function
$f\colon [0,T]\to S$ is defined by
\begin{align*}
  \norm{f}_{TV_2}^2:=\sup \sum_{k=0}^m \norm{f(t_k)-f(t_{k-1})}^2,
\end{align*}
where the supremum is taken over all partitions of $[0,T]$.

Let $U$ be a separable Banach space with norm $\norm{\cdot}$. The dual space is denoted by $U^\ast$ with dual 
pairing $\scapro{u}{u^\ast}$. The Borel $\sigma$-algebra in $U$ is denoted by $\Borel(U)$. 
For another separable Banach space $V$ the space of bounded, linear operators from $U$ to $V$ is 
denoted by $\L(U,V)$ equipped with the norm topology $\norm{\cdot}_{U\to V}$.

Let $(\Omega,\A,P)$ be a probability space. The space of equivalence classes of measurable
functions $f\colon\Omega\to U$ is denoted by $L_P^0(\Omega;U)$ and it is
equipped with the  topology of convergence in probability. The space of
equivalence classes of measurable functions whose $p$-th power has
finite integral is denoted by $L_P^p(\Omega;U)$ for $p\ge 1$.

\bigskip

\noindent
\textbf{Acknowledgements:} The authors thank the King's College London and FSU Jena for hospitality. 
The second named author acknowledges the EPSRC grant EP/I036990/1.

\section{The Skorokhod space}

In this section, we introduce the Skorokhod space and some of
its topologies. Let $V$ denote a separable Banach space.
For a fixed time $T>0$, the space of $V$-valued c{\`a}dl{\`a}g functions is denoted by $D([0,T];V)$. For each $f\in D([0,T];V)$
we define the set of discontinuities by
\begin{align*}
  J(f):=\{t\in (0,T]\colon  f(t-)\neq f(t)\}.
\end{align*}
The set $J(f)$ is countably finite. The jump size at $t$ is defined by $(\Delta f)(t)=f(t)-f(t-)$.
For two elements $v_1,v_2\in V$ we define the \emph{segment} as
the straight line between $v_1$ and $v_2$:
\begin{align*}
  \seg{v_1}{v_2}:=\{v\in V\colon  v=\alpha v_1 + (1-\alpha)v_2  \text{ for } \alpha \in [0,1]\}.
\end{align*}

In order to define a metric on $D([0,T]; V)$, the so-called \emph{(strong) $M_1$ metric},
we define for each $f\in D([0,T]; V)$ the {\em extended graph of $f$} by
\begin{align*}
  \Gamma(f):= \{(t,v)\in [0,T]\times V\colon  v\in \seg{f(t-)}{f(t)}\},
\end{align*}
where $f(0-):=f(0)$.
The projection of $\Gamma(f)$ to its spatial component in $V$ is given by
\begin{align*}
  \pi(\Gamma(f)):=\{v\in V\colon  (t,v)\in \Gamma(f)\quad\text{for some }t\in [0,T]\}.
\end{align*}
A total order relation on $\Gamma(f)$ is given by
\begin{align*}
  (t_1,v_1)\le (t_2,v_2)
  \;\Leftrightarrow\; \begin{cases}
    t_1< t_2 \quad\text{or} \\
    t_1=t_2\text{ and } \norm{f_1(t_1-)- v_1}\le \norm{f_1(t_1-)-v_2}.
    \end{cases}.
\end{align*}
A {\em parametric representation} of the extended graph of $f$ is a continuous, non-decreasing, surjective function
\begin{align*}
  (r,u):[0,1]\to \Gamma(f), \qquad
   (r,u)(0)=(0,f(0)), \; (r,u)(1)=(T,f(T)).
\end{align*}
Let $\Pi(f)$ denote the set of all parametric representations of $f$.

\subsection{Strong $M_1$ topology}\label{sse.strongtop}

For $f_1,f_2\in D([0,T]; V)$ we define
\begin{align*}
  d_{M}(f_1,f_2):=\inf\Big\{ \abs{r_1-r_2}_{\infty}\vee \norm{u_1-u_2}_{\infty} \colon
     (r_i,u_i)\in \Pi(f_i),\, i=1,2 \Big\}.
\end{align*}
As in the finite dimensional situation, cf.\ \cite[Theorem 12.3.1]{whitt02}, it follows that $d_M$ is a metric on $D([0,T];V )$, and we call it the {\em strong $M_1$ metric}. The metric space $\big(D([0,T];V),d_M\big)$ is separable but not complete.

Convergence of a sequence of functions
in the metric $d_M$ can be described by quantifying the oscillation of the functions. For
$v,v_1, v_2\in V$ the {\em distance from $v$ to  the segment between $v_1$ and $v_2$} is defined by
\begin{align*}
  M(v_1,v, v_2):= \inf_{\alpha\in [0,1]}\norm{v- (\alpha v_1 + (1-\alpha)v_2)}.
\end{align*}
The distance $M$ obeys for every   $v,v_1,v_2,v^\prime, v^\prime_1,v^\prime_2\in V$ the inequality
\begin{align}\label{eq.inM}
M(v_1,v,v_2)\le M(v^\prime_1,v^\prime, v^\prime_2)+ \norm{v-v^\prime}+\norm{v_1-v^\prime_1}+\norm{v_2-v^\prime_2},
\end{align}
and, instead of a triangular inequality, it satisfies
\begin{align}\label{eq.inM+}
M(v_1+v^\prime_1,v+v^\prime, v_2+v^\prime_2)
\le M(v_1,v,v_2)+ \norm{v^\prime}+ \big( \norm{v^\prime_1}\vee \norm{v^\prime_2}\big).
\end{align}
For functions $f,g\in D([0,T];V)$ and $0\le t_1\le t\le t_2\le T$  it follows from  \eqref{eq.inM+} that
\begin{align} \label{eq.f+ginfty}
& M\big(f(t_1)+g(t_1), f(t)+g(t), f(t_2)+g(t_2)\big)
\le M\big(f(t_1), f(t), f(t_2)\big) + 2 \norm{g}_\infty,
\intertext{and if $t_2-t_1\le \delta$ then}
\begin{split}\label{eq.f+gcont}
&  M\big(f(t_1)+g(t_1), f(t)+g(t), f(t_2)+g(t_2)\big)\\
&\qquad\qquad\qquad  \le M\big(f(t_1), f(t), f(t_2)\big) + \sup_{\substack{s_1,s_2\in [0,T]\\ \abs{s_2-s_1}\le \delta}}\norm{g(s_1)-g(s_2)}.
\end{split}
\end{align}
Define for $f\in D([0,T];V)$ and $\delta>0$ the oscillation function by
\begin{align*}
  M(f;\delta):= \sup\Big\{ M\Big(f(t_1), f(t),f(t_2)\Big)\colon  0\le t_1<t <t_2\le T \text{ and } t_2-t_1\le \delta\Big\}.
\end{align*}

\begin{lemma}\label{le.segandM}
Let $f$ be in $D([0,T];V)$ and let $0\le t_1\le t_2\le t_3\le T$ with $(t_i,v_i)\in \Gamma(f)$ for some $v_i\in V$ and  $i=1,2,3$.
If $t_3-t_1\le \delta$ for some $\delta>0$ then
\begin{align*}
  M(v_1,v_2,v_3)\le M(f;\delta).
\end{align*}
\end{lemma}
\begin{proof}
  Follows as Lemma  12.5.2 in Whitt \cite{whitt02}.
\end{proof}

\begin{lemma}\label{le.Mdisappears}
If $f\in D([0,T];V)$ then
$\displaystyle
  \lim_{\delta \searrow 0} M(f;\delta)=0.
$
\end{lemma}
\begin{proof}
  Follows as Lemma  12.5.3 in \cite{whitt02}.
\end{proof}

\begin{lemma}\label{le.sup+M}
  Let $f_n^\gamma, f^\gamma, f_n, f$, $n\in\N$, $\gamma> 0$, be  functions in $D([0,T];\R)$ satisfying
\begin{enumerate}
  \item[{\rm (i)}] $\displaystyle \lim_{n\to\infty}\left( \norm{f_n-f}_{\infty} +\limsup_{\gamma\to\infty}\norm{f_n^\gamma-f^\gamma}_\infty\right)= 0$;
 \item[{\rm (ii)}] $\displaystyle\lim_{\gamma\to\infty} f_n^\gamma = f_n$
   in $(D([0,T];\R),d_M)$ for all $n\in\N$.
\end{enumerate}
Then it follows that $\displaystyle\lim_{\gamma\to\infty} f^\gamma = f$
in $\big(D([0,T];\R),d_M\big)$.
\end{lemma}
\begin{proof}
Fix  $\epsilon>0$ and choose $n_0\in\N$ such that
\begin{align*}
   \norm{f_n-f}_{\infty}\le \epsilon\;\text{ and } \; \limsup_{\gamma\to\infty}\norm{f_n^\gamma-f^\gamma}_\infty\le  \epsilon
   \qquad\text{for all }n\ge n_0.
\end{align*}
Thus, there exists $\gamma_0=\gamma_0(n_0)$ such that
\begin{align*}
  \norm{f_{n_0}^\gamma-f^\gamma}_\infty \le 2\epsilon
  \quad\text{for all }\gamma\ge \gamma_0.
\end{align*}
 Condition (2) implies by part (iv)  in
\cite[Theorem 12.5.1]{whitt02} that there exists a dense subset $D\subseteq [0,T]$ including $0$ and $T$ such that
for each $t\in D$ there exists a $\gamma_1=\gamma_1(t,n_0)>0$ with
\begin{align}\label{eq.contint}
  \abs{f_{n_0}^\gamma (t)- f_{n_0}(t)}\le \epsilon \qquad\text{for all }\gamma\ge\gamma_1,
\end{align}
and that there exists a $\delta_0=\delta_0(n_0)>0$ such that
\begin{align}\label{eq.Mvan}
 \limsup_{\gamma\to\infty} M(f_{n_0}^\gamma, \delta)\le \epsilon
  \qquad\text{for all }\delta\le \delta_0.
\end{align}
Consequently, we can conclude from \eqref{eq.contint} for each $t\in D$
and  $\gamma\ge \max\{\gamma_0,\gamma_1\}$ that
  \begin{align*}
    \abs{f^\gamma(t)-f(t)}
    \le \abs{f^\gamma(t)-f_{n_0}^\gamma(t)} + \abs{f_{n_0}^\gamma(t)-f_{n_0}(t)} + \abs{f_{n_0}(t)-f(t)}
    \le 4 \epsilon.
  \end{align*}
Thus we have shown that
\begin{align}\label{eq.fgamma-cont}
  \lim_{\gamma\to\infty} f^\gamma(t)=f(t)
  \qquad\text{for all }t\in D.
\end{align}
It follows from  \eqref{eq.Mvan} for each $\delta\le \delta_0$ by inequality \eqref{eq.f+ginfty} that
\begin{align}\label{eq.fgamma-M}
 \limsup_{\gamma\to\infty}   M(f^\gamma,\delta)
  \le  \limsup_{\gamma\to\infty}  M(f_{n_0}^\gamma,\delta) + 2\limsup_{\gamma\to\infty} \norm{f_{n_0}^\gamma-f^\gamma}
  \le 3 \epsilon.
\end{align}
By \eqref{eq.fgamma-cont} and \eqref{eq.fgamma-M} a  final application of Theorem 12.5.1 in \cite{whitt02}  completes the proof.
\end{proof}

The metric space $(D([0,T];V), d_M)$ is not complete. However, one can define another metric
$\hat{d}_M$ on $D([0,T];V)$ such that $(D([0,T];V),\hat{d}_M)$ is complete and the two topological spaces
$(D([0,T];V),\,d_M)$ and $(D([0,T];V),\,\hat{d}_M)$ are homeomorphic, that is there exists a bijective
function $i\colon (D([0,T];V),\,d_M)\to (D([0,T];V),\,\hat{d}_M)$ such that both $i$ and its inverse
are continuous, see \cite[Theorem 12.8.1]{whitt02}. The last property, i.e.\ the existence of a homeomorphic mapping
between the metric spaces, is called the {\em topological  equivalence} of $(D([0,T];V),\,d_M)$ and
$(D([0,T];V),\,\hat{d}_M)$. In particular, this means that open, closed and compact sets are the same in both spaces but also
the spaces  of real-valued, continuous functions on $D([0,T];V)$  coincide
for both metrics. Moreover, since $(D([0,T];V),\,d_M)$ is separable,
the space $(D([0,T];V),\,\hat{d}_M)$ is also separable, and thus it is Polish, i.e.\ a topological
space which is metrisable as a complete separable space.

\subsection{Product $M_1$ topology}

For the product $M_1$ topology we assume that the Banach space $V$ has a Schauder basis
$e=(e_k)_{k\in\N}$ and that $(e_k^\ast)_{k\in\N}$ denotes the bi-orthogonal functionals.  Instead of equipping $D([0,T];V)$ with the strong $M_1$ topology we can consider the space as the Cartesian product space
$\prod_{k=1}^\infty D([0,T];\R)$ and equip it with the product metric
\begin{align*}
  d_M^{{\,e}}(f,g):= \sum_{k=1}^\infty \frac{1}{2^k}\frac{d_M(\scapro{f}{e_k^\ast},\scapro{g}{e_k^\ast})}{1+ d_M(\scapro{f}{e_k^\ast},\scapro{g}{e_k^\ast})}
  \qquad\text{for all }f,g\in D([0,T];V).
\end{align*}
The metric on the right hand side refers to the metric on the space $D([0,T];\R)$ introduced in the previous Section \ref{sse.strongtop}. Clearly, convergence in $d_M^{\,e}$ depends on the chosen Schauder basis $e$ of $V$.
Alternatively, we
can use the topological equivalent metric $\hat{d}_M$ on $D([0,T];\R)$ to define
the product metric
\begin{align*}
  \hat{d}_M^{{\,e}}(f,g):= \sum_{k=1}^\infty \frac{1}{2^k}\frac{\hat{d}_M(\scapro{f}{e_k^\ast},\scapro{g}{e_k^\ast})}{1+ \hat{d}_M(\scapro{f}{e_k^\ast},\scapro{g}{e_k^\ast})}
  \qquad\text{for all }f,g\in D([0,T];V).
\end{align*}
Since $(D([0,T];\R),\hat{d}_M)$ is a Polish space it follows that $(D([0,T];V),\hat{d}_M^{\,e})$ is Polish, too.
Analogously, we obtain that $(D([0,T];V),d_M^{\,e})$ is topological equivalent to $(D([0,T];V),\hat{d}_M^{\,e})$.
Recall, that the product topology is the topology of {\em point-wise convergence}, i.e.\ a sequence $(f_n)_{n\in\N}$
converges to $f$ in $(D([0,T];V),d_M^{\,e})$ if and only if for any $k\in\N$
\begin{align*}
  \lim_{n\to\infty} d_M\big(\scapro{f_n}{e_k^\ast},\scapro{f}{e_k^\ast}\big)=0.
\end{align*}

\subsection{Weak $M_1$ topology}

In an infinite dimensional Banach space $V$ there is a third mode of convergence in the $M_1$ sense.
A sequence $(f_n)_{n\in\N}\subseteq D([0,T];V)$ is said to {\em converge weakly} to $f\in D([0,T];V)$ if for all $v^\ast\in V^\ast$ we have
\begin{align*}
\lim_{n\to\infty} \scapro{f_n}{v^\ast}=\scapro{f}{v^\ast}  \qquad\text{in } \big(D([0,T];\R), d_M\big).
\end{align*}
Note, that if $V$ is infinite dimensional the induced topology is not metrisable.
The three different modes of convergence are related as shown by the following diagram:
\begin{align*}
 \text{strong $M_1$ }
  \Rightarrow
 \text{weak $M_1$ }
 \Rightarrow
  \text{product $M_1$.}
\end{align*}
The first implication follows from the fact that  $f_1,f_2\in D([0,T];V)$ obey the inequality
\begin{align*}
 d_M(\scapro{f_1}{v^\ast},\, \scapro{f_2}{v^\ast})\le (\norm{v^\ast}\vee 1) \, d_M(f_1,f_2)
\qquad\text{for all }v^\ast\in V^\ast.
\end{align*}
Since the product topology is the point-wise convergence it is immediate  that it is implied by  weak convergence.

If $V$ is finite dimensional it is known that the weak and strong topology coincide, see Theorem 12.7.2 in \cite{whitt02}. In the infinite dimensional situation the situation differs as illustrated by the following example.
\begin{example}
Let $V$ be an arbitrary Hilbert space with orthonormal basis $(e_k)_{k\in\N}$.
The functions $f_n \colon [0,T]\to V$ can be chosen as
$f_n(t):=e_n $ for all $t\in [0,T]$ and $n\in \N$. It follows that
\begin{align*}
\sup_{t\in [0,T]}\abs{\scapro{f_n(t)}{v}}=\abs{ \scapro{e_n}{v}}\to 0
\qquad\text{as $n\to \infty$ for all }v\in V,
\end{align*}
and thus, the sequence $(f_n)_{n\in\N}$ converges weakly to $0$ in $D([0,T];V)$. However, since $\norm{f_n(t)}=1$
for all $n\in \N$ and $t\in [0,T]$ it does not converge strongly.
\end{example}

\begin{example}
It is well known that in a finite dimensional Hilbert space $V$ convergence in the product topology does not imply strong convergence in the $M_1$ metric. Since in the finite dimensional situation, the strong $M_1$ topology coincides with the topology of weak convergence in $D([0,T];V)$, this is also an example that convergence in the product topology does not imply weak convergence.
\end{example}

Finally let us remark that our notion of the modes of convergence in the $M_1$ sense does not
coincide with the one in the literature such as \cite{whitt02}. There the product topology is called weak topology, whereas our weak topology does not have a name as it coincides with the strong topology in finite dimensional spaces. Since addition is not continuous in $D([0,T];V)$  our notion of weak convergence can not be confused with the usual weak topology in a linear topological vector space.

\subsection{Random variables in the Skorokhod space}

Let $\Borel(D)$ denote the Borel-$\sigma$-algebra generated by open sets in $(D([0,T];V),d_M)$. As in the finite dimensional situation,
see \cite[Theorem 11.5.2]{whitt02}, it can be shown that $\Borel(D)$  coincides with the $\sigma$-algebra, 
generated by the coordinate mappings
\begin{align*}
  \pi_{t_1,\dots, t_n}\colon D([0,T];V)\to V^n, \qquad \pi_{t_1,\dots, t_n}(f)=
  (f(t_1),\dots, f(t_n))
\end{align*}
for each $t_1,\dots, t_n\in [0,T]$ and $n\in\N$. The analogous result for the Skorokhod $J_1$ 
topology in separable metric spaces can be found in \cite{Jakubowski86}.
On the other hand, the Borel-$\sigma$-algebra  generated by open sets
in $(D([0,T];V),d_M^{\,e})$ equals the product of Borel-$\sigma$-algebras in $(D([0,T];\R),d_M)$. Consequently, both
Borel-$\sigma$-algebras, generated by open sets with respect to $d_M$ and $d_M^{\,e}$ coincide.

Let $(\Omega,\A,P)$ be a probability space and let $X:=(X(t)\colon t\in [0,T])$ be a $V$-valued stochastic process
with c{\`a}dl{\`a}g paths. Since the Borel-$\sigma$-algebra $\Borel(D)$ is generated by the coordinate mappings,
it follows that $X$ is a $D([0,T];V)$-valued random variables.

\section{Convergence in probability}\label{se.conv-prob}

\subsection{Strong topology}

In this section we consider the convergence in probability in the strong metric $d_M$ of stochastic processes
$(X_n)_{n\in\N}$ to a stochastic process $X$ in the space $D([0,T];V)$. If the stochastic processes
$X$ and $X_n$ have c{\`a}dl{\`a}g paths then $X$ and $X_n$ are $D([0,T);V)$-valued random variables.
Since $\big(D([0,T);V),d_M\big)$ is separable,  convergence in probability is well defined in the sense that
$(X_n)_{n\in\N}$ converges to $X$ in
probability in $\big(D([0,T);V),d_M\big)$ if
\begin{align*}
  \lim_{n\to\infty} P\Big( d_M(X_n,X)\ge \epsilon\Big)=0
  \qquad\text{for all }\epsilon>0.
\end{align*}

\begin{lemma}\label{le.stochconsup}
   A $V$-valued stochastically continuous stochastic process $(X(t)\colon t\in [0,T])$
   with c{\`a}dl{\`a}g trajectories obeys
   \begin{align*}
     \lim_{\delta\searrow 0} \sup_{t\in[0,T]} P\left(\sup_{\substack{ s\in [0,T] \\\abs{s-t}\le \delta}} \norm{X(t)-X(s)}\ge\epsilon\right)=0
     \qquad\text{for all }\epsilon>0.
   \end{align*}
\end{lemma}
\begin{proof}
  Define for each $\delta>0$ and $t\in [0,T]$ the random variable
  \begin{align*}
    Z(t,\delta):=\sup_{\substack{ s\in [0,T] \\\abs{s-t}\le \delta}} \norm{X(t)-X(s)}
  \end{align*}
and assume for a contradiction that there exist $\epsilon_1,\epsilon_2>0$, a sequence
$(\delta_n)_{n\in\N}\subseteq \Rp$ converging to 0, and a sequence $(t_n)_{n\in\N}\subseteq [0,T]$ such that
\begin{align*}
  \lim_{n\to\infty} P\big( Z(t_n,\delta_n)\ge \epsilon_1\big)\ge \epsilon_2.
\end{align*}
By passing to a subsequence if necessary we can assume that $t_n\to t_0$
for some $t_0\in [0,T]$. Then, for each $\delta>0$ there exists $n(\delta)\in\N$ such
that
\begin{align*}
  [t_n-\delta_n,t_n+\delta_n]\subseteq [t_0-\delta,t_0+\delta]
  \qquad\text{for all }n\ge n(\delta),
\end{align*}
which implies by the definition of $Z$ that
\begin{align*}
  Z(t_n,\delta_n)\le Z(t_0,\delta)\qquad \text{for all }n\ge n(\delta).
\end{align*}
Consequently, we obtain for every $\delta>0$ that there exists  $ n(\delta)\in\N$ such that
\begin{align}\label{eq.EZ>e}
  \epsilon_2\le E\big[ \1_{\{Z(t_n,\delta_n)\ge\epsilon_1\}}\big]
  \le   E\big[ \1_{\{Z(t_0,\delta)\ge \epsilon_1\}}\big]   \qquad\text{for all }n\ge n(\delta).
\end{align}
On the other hand, since $Z(t_0,\delta)\to \abs{\Delta X(t_0)}$ $P$-a.s. as $\delta\searrow 0$,
Lebesgue's theorem of dominated convergence implies that
\begin{align*}
0= P\big(\abs{\Delta X(t_0)}\ge \epsilon_1\big)
= E\left[\lim_{\delta\searrow 0} \1_{\{Z(t_0,\delta)\ge \epsilon_1\}}\right]
=\lim_{\delta\searrow 0} E\left[ \1_{\{ Z(t_0,\delta)\ge \epsilon_1\}}\right],
\end{align*}
which contradicts \eqref{eq.EZ>e}.
\end{proof}

\begin{theorem}\label{th.stochconv=fdd+M}
For $V$-valued, stochastically continuous stochastic processes $(X(t)\colon t\in [0,T])$ and
 $(X_n(t)\colon t\in [0,T])$, $n\in\N$, with c{\`a}dl{\`a}g trajectories the following are equivalent:
\begin{enumerate}
\item[{\rm (a)}] $X_n\to X$ in probability in $\big(D([0,T];V),\, d_M\big)$ as $n\to\infty$.
\item[{\rm (b)}] the following two conditions are satisfied:
  \begin{enumerate}
  \item[{\rm (i)}] for every $t\in [0,T]$ we have $\lim_{n\to\infty} X_n(t)=X(t)$ in probability
  \item[{\rm (ii)}] for every $\epsilon>0$ the oscillation function obeys
  \begin{align}
  \label{eq:M}
      \lim_{\delta\searrow 0} \limsup_{n\to\infty} P\big(M(X_n,\delta)\ge\epsilon\big)=0.
  \end{align}
\end{enumerate}
\end{enumerate}
\end{theorem}

\begin{proof}
(a)$\,\Rightarrow\,$ (b)
To establish property (i), let $t\in [0,T]$
and $\epsilon_1,\epsilon_2>0$ be given.
Lemma \ref{le.stochconsup} guarantees that there exists
$\delta>0$ such that the set
\begin{align*}
  E(\epsilon_1,\delta):=\bigg\{\omega\in \Omega\colon  \sup_{\substack{s\in [0,T]\\ \abs{s-t}\le \delta}}
   \norm{X(t)(\omega)-X(s)(\omega)}\le \epsilon_1\bigg\},
\end{align*}
satisfy $
  P\big(E(\epsilon_1,\delta)\big)\ge 1-\frac{\epsilon_2}{2}$.
By the assumed condition (a) there is $n_0\in\N$ such that
\begin{align*}
  P\big(d_M(X_n,X)< (\epsilon_1\wedge \delta)\big)\ge 1-\frac{\epsilon_2}{2}
  \qquad\text{for all }n\ge n_0.
\end{align*}
Consequently, the set
\begin{align*}
  F(\epsilon_1,\delta,n):=  E(\epsilon_1,\delta) \cap \{d_M(X_n,X)< (\epsilon_1\wedge \delta)\}
\end{align*}
 satisfies $P\big(F(\epsilon_1,\delta,n)\big)\ge 1-\epsilon_2$ for every $n\ge n_0$.
Define for $\omega\in F(\epsilon_1,\delta, n)$ the functions $f_n:=X_n(\cdot)(\omega)$ and $f:=X(\cdot)(\omega)$. It follows that there are parametric representations
$(r,u)\in \Pi(f)$ and $(r_n, u_n)\in \Pi(f_n)$  satisfying
\begin{align}\label{eq.in1utau}
  \abs{r-r_n}_\infty \vee \norm{u-u_n}_\infty \le (\epsilon_1\wedge \delta )
  \qquad\text{for all }n\ge n_0.
\end{align}
For every $t \in [0,T]$ denote $\tau ,\tau_{n}\in [0,1]$ for $n\ge n_0$ such that
\begin{align*}
  (t,f(t))&=(r(\tau),u(\tau))
  \quad \text{ and }\quad
    (t,f_n(t))=(r_n(\tau_{n}),u_n(\tau_{n})). 
\end{align*}
Since $u(\tau_{n})\in \seg{f(r(\tau_{n})-)}{f(r(\tau_{n}))}$ for every 
$n\ge n_0$,
there is $\alpha_{n} \in [0,1]$ such that
\begin{align*}
  u(\tau_{n})=\alpha_{n} f(r(\tau_{n})-)+ (1-\alpha_{n})f(r(\tau_{n})).
\end{align*}
Since $t=r_n(\tau_{n})$ and $\abs{r(\tau_{n})-r_n(\tau_{n})}\le \delta $
for all $n\ge n_0$ by \eqref{eq.in1utau}, we have $\abs{r(\tau_{n})-t}\le \delta $.
Another application of inequality \eqref{eq.in1utau} implies
\begin{align} \label{eq.in2utau}
  \norm{u(\tau_{n})-u(\tau)}
  &= \norm{\alpha_{n} f(r(\tau_{n})-)+ (1-\alpha_{n})f(r(\tau_{n}))-f(t)}\notag\\
  &\le \sup_{\substack{s\in [0,T]\\ \abs{s-t}\le \delta}}\sup_{\alpha\in [0,1]}
   \norm{\alpha f(s-)+ (1-\alpha)f(s) - f(t)}\notag\\
  & =  \sup_{\substack{s\in [0,T]\\ \abs{s-t}\le \delta}}\sup_{\alpha\in [0,1]}
   \norm{\alpha \big(f(s-)-f(t)\big)+ (1-\alpha)\big(f(s) - f(t)\big)}\notag\\
  &\le \sup_{\substack{s\in [0,T]\\ \abs{s-t}\le \delta}} \norm{f(s-)-f(t)}
   +  \sup_{\substack{s\in [0,T]\\ \abs{s-t}\le \delta}} \norm{f(s)-f(t)}\notag\\
  &\le 2\epsilon_1.
\end{align}
Inequalities \eqref{eq.in1utau} and \eqref{eq.in2utau} imply
for every $n\ge n_0$ that
\begin{align*}
  \norm{f(t)-f_n(t)}
  &= \norm{u(\tau)- u_n(\tau_{n})}\\
  &\le \norm{u(\tau)-u(\tau_{n})} + \norm{u(\tau_{n})-u_n(\tau_{n})}\le 3\epsilon_1,
\end{align*}
which establishes Condition (i) in (b).

In order to show Condition (ii) fix some $\epsilon_1,\epsilon_2>0$. Lemma \ref{le.Mdisappears} guarantees
that there exists $\delta_0>0$ such that the set
\begin{align*}
  G(\epsilon_1,\delta):=\left\{\omega\in\Omega\colon  M(X(\omega),\delta)\le\epsilon_1\right\}
\end{align*}
satisfies $P(G(\epsilon_1,\delta))\ge 1-\tfrac{\epsilon_2}{2}$ for all $\delta\in [0,\delta_0]$.
For each $\delta>0$ there is  $n_0\in\N$  by the assumed condition (a) such that
\begin{align*}
  P\big(d_M(X_n,X)< (\epsilon_1\wedge \delta)\big)\ge 1-\frac{\epsilon_2}{2}\qquad
  \text{for all }n\ge n_0.
\end{align*}
Together we obtain for every $\delta\in [0,\delta_0]$
\begin{align*}
  \liminf_{n\to\infty} P\big( G(\epsilon_1,\delta)\cap \{d_M(X_n,\delta)<(\epsilon_1\wedge \delta)\} \big)
  \ge 1-\epsilon_2.
\end{align*}
Fix $\omega\in G(\epsilon_1,\delta)\cap \{d_M(X_n,\delta)<(\epsilon_1\wedge \delta)\}$ and define $f:=X(\cdot)(\omega)$
and $f_n:=X_n(\cdot)(\omega)$. It follows that there are parametric representations $(r,u)\in \Pi(f)$ and $(r_n,u_n)\in \Pi(f_n)$
satisfying
\begin{align*}
 \abs{r-r_n}_\infty\vee \norm{u-u_n}_\infty  \le (\epsilon_1\wedge \delta)\qquad  \text{for all }n\ge n_0.
\end{align*}
For every $0\le t_1\le t_2\le t_3 $ denote $\tau_i,\tau_{i,n}\in [0,1]$ such that $(t_i,f(t_i))=(r(\tau_i),u(\tau_i))$ and
$(t_i,f_n(t_{i,n}))=(r_n(\tau_{i,n}), u_n(\tau_{i,n}))$ for $i=1,2,3$.
Inequality \eqref{eq.inM} and Lemma \ref{le.segandM} imply  for every $n\ge n_0$ that
\begin{align*}
M\big(f_n(t_1),f_n(t_2),f_n(t_3)\big)
&= M\big(u_n(\tau_{1,n}), u_n(\tau_{2,n}), u_n(\tau_{3,n})\big)\\
&\le M\big(u(\tau_{1,n}), u(\tau_{2,n}), u(\tau_{3,n})\big)+ 3 \norm{u-u_n}_\infty\\
&\le M(f,\delta) + 3\norm{u-u_n}_\infty\\
&\le \epsilon_1+3\epsilon_1=4\epsilon_1,
\end{align*}
which completes the proof of the implication (a)$\,\Rightarrow\,$(b).

\noindent
(b) $\Rightarrow$ (a).
Let $\epsilon_1, \epsilon_2>0$ be fixed. Define for each $n\in \N$ and $\delta>0$ the sets
\begin{align*}
G(\epsilon_1, \delta )&:= \Big\{\omega\in \Omega\colon  M(X(\omega), \delta)< \tfrac{\epsilon_1}{512}\Big\},\\
  G_n(\epsilon_1,\delta)&:= \Big\{\omega\in\Omega\colon  M(X_n(\omega),\delta)<  \tfrac{\epsilon_1}{512}\Big\}.
\end{align*}
Condition (ii) guarantees that there exist $\delta_1>0$  and $n_1\in \N$
such that
\begin{align*}
  \limsup_{n\to\infty} P\big(G_n^c(\epsilon_1,\delta)\big)\le \frac{\epsilon_2}{8}
  \qquad\text{for all }\delta \in [0,\delta_1].
\end{align*}
Consequently, for each $\delta \in [0,\delta_1]$ there exists $n_1=n_1(\delta)$ such that
\begin{align}\label{eq.probGn}
  \sup_{n\ge n_1} P\big(G_n^c(\epsilon_1,\delta)\big)\le \frac{\epsilon_2}{4},
\end{align}
whereas Lemma \ref{le.Mdisappears} implies that there exist  $\delta_2>0$ such that
\begin{align}\label{eq.probG}
  P\big(G(\epsilon_1,\delta)\big)\ge 1-\frac{\epsilon_2}{4}   \qquad\text{for all }\delta \in [0,\delta_2].
\end{align}
Define for $c>0$ the set
\begin{align*}
  B(c):= \Big\{\omega\in \Omega\colon  \norm{X(\omega)}_\infty \le c-1 \Big\}.
\end{align*}
Since $X$ is a random variable with values in $D([0,T];V)$ there exists $c>0$ such that
\begin{align}\label{eq.boundc}
  P\big(B(c)\big)\ge 1-\frac{\epsilon_2}{4}.
\end{align}
Choose a partition $\pi=(t_i)_{i=0}^m$ of the interval $[0,T]$ such that
\begin{align*}
0=t_0< t_1<\dots < t_m=T \quad\text{and}\quad  \max_{i\in \{1,\dots, m\}}\abs{t_i-t_{i-1}}\le \min\{\delta_1,\delta_2,\tfrac{\epsilon_1}{16}\},
\end{align*}
and define the set
\begin{align*}
  F_n(\epsilon_1,\pi):=\Big\{ \omega\in \Omega\colon  \max_{i=1,\dots, m} \norm{X_n(t_i)(\omega)-X(t_i)(\omega)}< \tfrac{\epsilon_1}{512}\Big\}. \end{align*}
Condition (i) guarantees that there exists  $n_2\in \N$ such that
\begin{align}\label{eq.probFn}
  \sup_{n\ge n_2} P\Big( F_n^c(\epsilon_1,\pi)\Big)\le \frac{\epsilon_2}{4}.
\end{align}
It follows  from \eqref{eq.probGn} to  \eqref{eq.probFn} that for $\delta:=\delta_1\wedge \delta_2$
the set $E_n(\epsilon_1,\delta,c,\pi):=G_n(\epsilon_1,\delta)\cap G(\epsilon_1,\delta)\cap B(c)\cap F_n(\epsilon_1,\pi)$
obeys
\begin{align*}
  P\Big( E_n(\epsilon_1,\delta,c,\pi)\Big)\ge 1-\epsilon_2
  \qquad\text{for all } n\ge n_1 \vee n_2.
\end{align*}
For $\omega\in E_n(\epsilon_1,\delta,c,\pi)$ define $f_0(\cdot):=X(\cdot)(\omega)$
 and
$f_n(\cdot):= X_n(\cdot)(\omega)$. Let $N$ denote the integers $\{n_1\vee n_2,\dots\}$ and
$N_0$ the union $N\cup \{0\}$.
For $n\in N_0$ and $i\in\{1,\dots, m\}$ let $\Gamma_i^n$ be the graph of $f_n$ between $(t_{i-1},f_n(t_{i-1}))$ and $(t_i,f_n(t_i))$.
By defining $d_i$ to be the smallest integer larger than
$
  \norm{f_0(t_{i-1})-f_0(t_i)}\frac{16}{\epsilon_1}
$
we can divide the segment $\seg{f_0(t_{i-1})}{f_0(t_i)}$ in
equidistant points
\begin{align*}
  \xi_{i,j}:=f_0(t_{i-1})+ \alpha_{i,j}(f_0(t_i)-f_0(t_{i-1}))
  \quad\text{for } \alpha_{i,j}:=\frac{j}{d_i}, \quad j=0,\dots, d_i.
\end{align*}
We claim that for each $i\in \{1,\dots, m\}$  the balls
$
  B_{i,j}:=\big\{ h\in V\colon  \norm{h-\xi_{i,j}}<\tfrac{\epsilon_1}{25}\big\}
$
covers each of the graphs $\Gamma_n^i$ for $n\in N_0$, i.e.
\begin{align}\label{eq.coverGamma}
\pi(\Gamma_n^i)\subseteq \bigcup_{j=0}^{d_i} B_{i,j}
\qquad\text{for all }n\in N_0.
\end{align}
Indeed, let $(t,h)\in \Gamma^n_i$ be of the form $h=\alpha f_n(t-)+(1-\alpha)f_n(t)$  for some $\alpha\in [0,1]$.
Since $t\in [t_{i-1},t_i]$ and $t_i-t_{i-1}\le \delta_2$ it follows from the definition of $M(f_n,\delta_2)$ that
there exists $\ell_n,r_n\in \seg{f_n(t_{i-1})}{f_n(t_i)}$ such that
\begin{align*}
  \norm{f_n(t-)- \ell_n}\le \tfrac{\epsilon_1}{512}
\quad\text{and}\quad  \norm{f_n(t)- r_n}\le \tfrac{\epsilon_1}{512}
\qquad\text{for all }n\in N_0.
\end{align*}
Since $u_n:=\alpha\ell_n + (1-\alpha) r_n\in \seg{f_n(t_{i-1})}{f_n(t_i)}$
we have $M(f_n(t_{i-1}),u_n, f_n(t_i))=0$. Inequality \eqref{eq.inM} implies that
\begin{align*}
  M(f_0(t_{i-1}),u_n,f_0(t_i))
  &\le  M(f_n(t_{i-1}),u_n,f_n(t_i)) +\norm{f_0(t_{i-1})-f_n(t_{i-1})} + \norm{f_0(t_i)-f_n(t_i)}\\
  &\le 0 + 2\max_{i\in\{0,\dots, m\}}\norm{f_n(t_i)-f_0(t_{i})}\\
  &<  2\frac{\epsilon_1}{512}.
\end{align*}
Consequently, there exists $u_0\in \seg{f_0(t_{i-1})}{f_0(t_i)}$ such that
$
  \norm{u_n-u_0}\le \frac{2\epsilon_1}{512}.
$
(If $n=0$ we can choose $u_0=u_n$.)
Since $u_0\in \seg{f_0(t_{i-1})}{f_0(t_i)}$ we can choose the closest node $\xi_{i,j}$ for
some $j=0,\dots, d_i$ such that
$\norm{u_0-\xi_{i,j}}\le \frac{\epsilon_1}{32}$. It follows
\begin{align*}
  \norm{h-\xi_{i,j}}
 & =\norm{\big(\alpha f_n(t-)+(1-\alpha)f_n(t)\big) -\xi_{i,j}}\\
&\le \alpha \norm{f_n(t-)- \ell_n}
  + (1-\alpha) \norm{f_n(t)- r_n}
  + \norm{u_n-u_0}+ \norm{u_0 - \xi_{i,j}}\\
&\le \frac{\epsilon_1}{512} + \frac{2\epsilon_1}{512}+\frac{\epsilon_1}{32}
\le \frac{\epsilon_1}{25},
\end{align*}
which shows \eqref{eq.coverGamma}.

In the following, we define for each $i\in \{1,\dots, m\}$ and $n\in N_0$ an ordered
sequence of points
\begin{align*}
  \big((r_{i,0}^n,z_{i,0}^n),\dots, (r_{i,m_i}^n,z_{i,m_i}^n)\big)
\in\big( \Gamma^n_i\times \dots \times\Gamma^n_i\big),
\end{align*}
for some $m_i\in \N$, independent of $n$, such that they satisfy for every $j=1,\dots, m_i$:
\begin{align}\label{eq.inznrn}
\sup_{(z,r)\in \Gamma^{n}_{i,j}}\max \left\{ \norm{z-z_{i,j-1}^n},\norm{z-z_{i,j}^n},
 \abs{r-r_{i,j-1}^n}, \abs{r-r_{i,j}^n}\right\}\le \frac{\epsilon_1}{4},
\end{align}
where $\Gamma^n_{i,j}:=\{(r,z)\in \Gamma^n_i\colon  (r_{i,j-1}^n,z_{i,j-1}^n, )
\le (r,z)\le (r_{i,j}^n,z_{i,j}^n)\}$.

If $d_i=1$ we define $m_i=1$ and for every  $n\in N_0$ the points:
\begin{align*}
   (r_{i,0}^n,z_{i,0}^n)&:= (t_{i-1},f_n(t_{i-1})), \quad  (r_{i,1}^n,z_{i,1}^n):= (t_{i},f_n(t_{i})).
\end{align*}
It follows from \eqref{eq.coverGamma} that for each $(r,z)\in \Gamma^n_i$ there is $k\in\{0,1\}$ such that
  $z\in B_{i,k}$. For $k=0$ this results in
\begin{align}\label{eq.di=1}
 \norm{z-z_{i,0}^n}\vee \norm{z-z_{i,1}^n}
 &\le \norm{z-\xi_{i,0}}+ \norm{\xi_{i,0}-\xi_{i,1}}+ \norm{\xi_{i,1}-f_n(t_i)}\notag\\
 &\le \frac{\epsilon_1}{25}+\frac{\epsilon_1}{16}+\frac{\epsilon_1}{512}\1_{ N}(n)\le \frac{\epsilon_1}{4},
\end{align}
and analogously for $k=1$. Since each $r\in [r_{i,0}^n,r_{i,1}^n]$
satisfies $\abs{r-r_{i,j}^n}\le \abs{r_{i,0}^n-r_{i,1}^n}\le \tfrac{\epsilon_1}{16}$
for $j\in \{0,1\}$ we obtain the inequality \eqref{eq.inznrn}.

If $d_i=2$ we define $m_i=3$ but we distinguish two cases. Firstly, assume that $\pi(\Gamma^n_i)\subseteq B_{i,0}\cup B_{i,2}$. Then we define for each $n\in N_0$ the points
\begin{align*}
   (r_{i,0}^n,z_{i,0}^n,)&:= (t_{i-1},f_n(t_{i-1})), \quad  (r_{i,3}^n,z_{i,3}^n):= (t_{i},f_n(t_{i}))
\end{align*}
and we choose $z_{i,1}^n,z_{i,2}^n\in \pi(\Gamma^n_i)\cap B_{i,0}\cap B_{i,2}$
and $r_{i,1}^n, r_{i,2}^n\in [0,1]$ such that
\begin{align*}
  (r_{i,0}^n,z_{i,0}^n ) < (r_{i,1}^n,z_{i,1}^n )< (r_{i,2}^n,z_{i,2}^n )<(r_{i,3}^n,z_{i,3}^n ).
\end{align*}

In the case $\pi(\Gamma_n^i)\not\subseteq B_{i,0}\cup B_{i,2}$ we define for every $n\in N_0$ the points
\begin{align*}
  (r_{i,0}^n,z_{i,0}^n)&:=(t_{i-1},f_n(t_{i-1})),\\
  (r_{i,1}^n,z_{i,1}^n)&:=\inf\{(r,z)\in \Gamma_{n}^i\colon  (r,z)>(r_{i,0}^n,z_{i,0}^n) \text{ and } z\in \partial B_{i,1}\}, \\
    (r_{i,2}^n,z_{i,2}^n)&:=\inf\{(r,z)\in \Gamma_{n}^i\colon  (r,z)>(r_{i,1}^n,z_{i,1}^n) \text{ and } z\in \partial B_{i,2}\},\\
  (r_{i,3}^n,z_{i,3}^n)&:=(t_{i},f_n(t_{i})).
\end{align*}

If $d_i\ge 3$ we define $m_i=d_i+1$. Since $\norm{\xi_{i,j}-\xi_{i,j-1}}\ge \tfrac{d_i-1}{d_i}\frac{\epsilon_1}{16}>
\frac{\epsilon_1}{25}$ we have $B_{i,j}\cap B_{i,j+2}=\emptyset$ for all
$j=1,\dots, d_i$. Thus, we can define the following increasing sequence:
\begin{align*}
    (r_{i,0}^n,z_{i,0}^n)&:=(t_{i-1},f_n(t_{i-1}) ),\\
  (r_{i,j}^n,z_{i,j}^n)&:=\inf\{(r,z)\in \Gamma^{n}_i\colon  (r,z)>(r_{i,j-1}^n,z_{i,j-1}^n) \text{ and } z\in \partial B_{i,j}\}, \quad j=1,\dots, d_i, \\
  (r_{i,m_i}^n,z_{i,m_i}^n)&:=( t_{i},f_n(t_{i})).
\end{align*}
In both cases for $d_i=2$ and in the case $d_i\ge 3$ it follows for each $n\in N$ that
\begin{align}\label{eq.z0d}
\begin{split}
  \norm{z_{i,0}^n-\xi_{i,0}}&=\norm{f_n(t_{i-1})-f_0(t_{i-1})}< \frac{\epsilon_1}{512}, \\
  \norm{z_{i,m_i}^n-\xi_{i,d_i}}&=\norm{f_n(t_{i})-f_0(t_i)}< \frac{\epsilon_1}{512}.
\end{split}
\end{align}
Consequently,  $z_{i,0}^n\in B_{i,0}$ and $z_{i,m_i}^n\in B_{i,d_i}$ and thus
$z_{i,0}^n, z_{i,1}^n\in \bar{B}_{i,0}$ and $z_{i,m_i-1}^n,z_{i,m_i}^n \in \bar{B}_{i,d_i}$
 for every $n\in N_0$.
Since  $z_{i,j-1}^n, z_{i,j}^n \in \bar{B}_{i,j-1}$ for all $j=2,\dots, m_i-1$ by construction,
we obtain
\begin{align}\label{eq.distancezs}
  \norm{z_{i,j-1}^n-z_{i,j}^n}\le 2 \frac{\epsilon_1}{25}
  \qquad\text{for all }j\in \{1,\dots, m_i\},\, n\in N_0.
\end{align}
If $(r,z)\in\Gamma_{i,j}^n$ for some $j\in \{1,\dots, m_i\}$ and $n\in N_0$ then
$M(z_{i,j-1}^n,z, z_{i,j}^n)< \tfrac{\epsilon_1}{512}$ since $\abs{r_{i,j-1}^n-r_{i,j}^n}
\le \abs{t_{i-1}-t_i}\le \delta_1$. Thus, there exists $z_0\in \seg{z_{i,j-1}^n, z_{i,j}^n}$ such
that $\norm{z-z_0}\le \tfrac{\epsilon_1}{512}$. Together with \eqref{eq.distancezs} it follows
for each $(r,z)\in \Gamma_{i,j}^n$ and $k\in \{j-1,j\}$
for $j={1,\dots, m_i}$ that
\begin{align*}
  \norm{z-z_{i,k}^n}
  \le \norm{z-z_0}+ \norm{z_0-z_{i,k}^n}
  \le \norm{z-z_0}+ \norm{z_{i,j-1}^n-z_{i,j}^n}
  \le \frac{\epsilon_1}{512}+ 2\frac{\epsilon_1}{25}\le \frac{\epsilon_1}{4}.
\end{align*}
Since we also have that
\begin{align*}
  \abs{r-r_{i,k}^n}\le \abs{t_i-t_{i-1}}\le \frac{\epsilon_1}{16},
\end{align*}
we obtain \eqref{eq.inznrn}.

The constructed sequence exhibits a further property: since for every $i\in\{1,\dots, m\}$ and $n\in N$ the points $z_{i,j}^0$ and $z_{i,j}^n$ are in the same closed ball $\bar{B}_{i,j}$
for $j\in \{0,m_i\}$   by \eqref{eq.z0d} and
for $j\in \{1,\dots, m_i-1\}$ by construction , it follows that
\begin{align*}
    \sup_{\substack{i\in \{1,\dots, m\}\\ j\in \{0,\dots, m_i\}}}  \left\{ \norm{z_{i,j}^0-z_{i,j}^n}\right\}\le 2\frac{\epsilon_1}{25}.
\end{align*}
Since $\abs{r_{i,j}^0-r_{i,j}^n}\le \abs{t_i-t_{i-1}}\le \tfrac{\epsilon_1}{16}$ we obtain that
\begin{align}\label{eq.inzrznrn}
    \sup_{\substack{i\in \{1,\dots, m\}\\ j\in \{0,\dots, m_i\}}}\max  \left\{ \norm{z_{i,j}^0-z_{i,j}^n},\abs{r_{i,j}^0-r_{i,j}^n}\right\}\le 2\frac{\epsilon_1}{25}.
\qquad\text{for all }n\in N.
\end{align}
By gluing together we obtain for each $n\in N_0$ an ordered sequence
\begin{align*}
  \big((r_{1,0}^n,z_{1,0}^n),\dots, (r_{1,m_1}^n,z_{1,m_1}^n),
   (r_{2,0}^n,z_{2,0}^n),\dots, ( r_{m,m_m}^n,z_{m,m_m}^n)\big) \in
   (\Gamma_n\times \dots \times \Gamma_n),
\end{align*}
satisfying the inequalities \eqref{eq.inznrn} and \eqref{eq.inzrznrn}.
It follows as in the proof of the implication $(vi)\Rightarrow (i)$ of Theorem 12.5.1 in \cite{whitt02}, that one can define
for every parametric representation $(r,u)\in \Pi(f_0)$ a parametric representation $(r_n,u_n)\in \Pi(f_n)$ such that
\begin{align*}
 \abs{r-r_n}_\infty \vee \norm{u-u_n}_\infty  \le 2\frac{\epsilon_1}{4}+ \frac{2\epsilon_1}{25}
  \qquad\text{for all }n\in N.
\end{align*}
Thus, we have shown that for each $\omega\in E_n(\epsilon_1,\delta,c,\pi)$ we have
\begin{align*}
  d_M(X(\omega),X_n(\omega))\le 2\frac{\epsilon_1}{4}+ \frac{2\epsilon_1}{25}
    \qquad\text{for all }n\in N,
\end{align*}
which completes the proof.
\end{proof}


\subsection{Product topology}\label{se.product-prob}

In this part we equip the space $D([0,T];V)$ with the product topology $d_M^{{\,e}}$
for a fixed Schauder basis $e:=(e_k)_{k\in\N}$ of $V$ with bi-orthogonal sequence $(e_k^\ast)_{k\in\N}$,  and we
 consider the convergence in probability of stochastic processes $(X_n)_{n\in\N}$ to a
stochastic process $X$. For stochastic process $X$ and $(X_n)_{n\in\N}$ with c{\`a}dl{\`a}g trajectories we  say that $(X_n)_{n\in\N}$ converges to $X$ in
probability in $\big(D([0,T];V),d_M^{\,e}\big)$ if
\begin{align*}
  \lim_{n\to\infty} P\Big( d_M^{\,e} (X_n,X)\ge \epsilon\Big)=0
  \qquad\text{for all }\epsilon>0.
\end{align*}
Since the product topology corresponds to point-wise convergence, the stochastic processes $(X_n)_{n\in\N}$ converges to $X$ in probability in $\big(D([0,T];V),d_M^{\,e}\big)$ if and only
if for every $k\in\N$
\begin{align}\label{eq.prod=point}
  \lim_{n\to\infty} P\Big( d_M\big(\scapro{X_n}{e_k^\ast},\scapro{X}{e_k^\ast}\big)\ge \epsilon\Big)=0
  \qquad\text{for all }\epsilon>0,
\end{align}
see \cite[Lemma 4.4.4]{Kallenberg-02}.
Consequently, we obtain as an analogue of Theorem \ref{th.stochconv=fdd+M}:
\begin{corollary}
\label{co.productcon=M+ffd}
Let $(e_k)_{k\in\N}$ be a Schauder basis of $V$ with bi-orthogonal sequence $(e_k^\ast)_{k\in\N}$.
For $V$-valued, stochastically continuous stochastic processes $(X(t)\colon t\in [0,T])$ and
 $(X_n(t)\colon t\in [0,T])$, $n\in\N$, with c{\`a}dl{\`a}g trajectories the following are equivalent:
\begin{enumerate}
\item[{\rm (a)}] $X_n\to X$ in probability in $\big(D([0,T];V),\, d_M^{\,e}\big)$ as $n\to\infty$;
\item[{\rm (b)}] the following two conditions are satisfied for every $k\in\N$:
  \begin{enumerate}
  \item[{\rm (i)}] for every $t\in [0,T]$ we have
  $\lim_{n\to\infty}\scapro{X_n(t)}{e_k^\ast}= \scapro{X(t)}{e_k^\ast} $ in probability;
  \item[{\rm (ii)}] for every $\epsilon>0$ the oscillation function obeys
  \begin{align*}
      \lim_{\delta\searrow 0} \limsup_{n\to\infty} P\big(M(\scapro{X_n}{e_k^\ast},\delta)\ge\epsilon\big)=0.
  \end{align*}
\end{enumerate}
\end{enumerate}
\end{corollary}
\begin{proof}
  Follows immediately from Theorem \ref{th.stochconv=fdd+M} and \eqref{eq.prod=point}.
\end{proof}

\subsection{Weak topology}

Recall that the weak $M_1$ topology in an infinite dimensional Hilbert space is not metrisable.
A sequence $(X_n)_{n\in\N}$ of stochastic processes $(X_n)_{n\in\N}$ with trajectories in  $D([0,T];V)$ is said to {\em converge weakly in $M_1$ in probability} to a process $X$ with trajectories in $D([0,T];V)$
if for all $v^\ast\in V^\ast$ we have
\begin{align*}
\lim_{n\to\infty} \scapro{X_n}{v^\ast}=\scapro{X}{v^\ast}
 \qquad\text{in probability in $\big(D([0,T];\R),d_M\big)$.}
\end{align*}
Equivalently, by using the metric $d_M$ in $D([0,T];\R)$, this convergence takes place if and only if for each $v^\ast\in V^\ast$ we have
\begin{align}\label{eq.weak=point}
  \lim_{n\to\infty} P\Big( d_M\big(\scapro{X_n}{v^\ast},\scapro{X}{v^\ast}\big)\ge \epsilon\Big)=0
  \qquad\text{for all }\epsilon>0.
\end{align}
By comparing \eqref{eq.weak=point} with \eqref{eq.prod=point} one can colloquially describe
the difference between convergence in the weak sense and in the product topology by testing the
one-dimensional projections either with all elements, i.e.\ $\scapro{X_n}{v^\ast}$
for all $v^\ast\in V^\ast$, or only with the bi-orthogonal elements of $V^\ast$, i.e.\ $\scapro{X_n}{e_k^\ast}$
for all $k\in\N$.  Clearly, the first one is independent of the chosen basis.
\begin{corollary}\label{co.weakcon=M+ffd}
For $V$-valued, stochastically continuous stochastic processes $(X(t)\colon t\in [0,T])$ and
 $(X_n(t)\colon t\in [0,T])$, $n\in\N$, with c{\`a}dl{\`a}g trajectories the following are equivalent:
\begin{enumerate}
\item[{\rm (a)}] $X_n\to X$ weakly in $M_1$ in probability in $D([0,T];V)$ as $n\to\infty$;
\item[{\rm (b)}] the following two conditions are satisfied for every $v^\ast\in V^\ast$:
  \begin{enumerate}
  \item[{\rm (i)}] for every $t\in [0,T]$ we have $\lim_{n\to\infty} \scapro{X_n(t)}{v^\ast}=\scapro{X(t)}{v^\ast}$
    in probability;
  \item[{\rm (ii)}] for every $\epsilon>0$ the oscillation function obeys
  \begin{align*}
      \lim_{\delta\searrow 0} \limsup_{n\to\infty} P\big(M(\scapro{X_n}{v^\ast},\delta)\ge\epsilon\big)=0.
  \end{align*}
\end{enumerate}
\end{enumerate}
\end{corollary}
\begin{proof}
  Follows immediately from Theorem \ref{th.stochconv=fdd+M} and \eqref{eq.weak=point}.
\end{proof}

\begin{remark}\label{re.path-cylindrical}
In this part we always require that the considered stochastic processes have c{\`a}dl{\`a}g paths in the Hilbert space $V$. If $V$ is infinite dimensional this might be a too restrictive assumption. In fact the definition of weak convergence  only requires that
the stochastic processes have {\em cylindrical c{\`a}dl{\`a}g trajectories}, that is
\begin{align*}
  \big(\scapro{X(t)}{v^\ast}\colon t\in [0,T]\big), \qquad
   \big(\scapro{X_n(t)}{v^\ast}\colon t\in [0,T]\big), \quad n\in\N,
\end{align*}
have c{\`a}dl{\`a}g trajectories for all $v^\ast\in V^\ast$. Since Corollary \ref{co.weakcon=M+ffd} is just
proved by the application of Theorem \ref{th.stochconv=fdd+M} to these real-valued stochastic processes we could easily soften our assumption on the path regularities of the considered stochastic processes accordingly.

The same comment applies to convergence in the product topology $(D([0,T];V),d_M^{\,e})$
for a Schauder basis $e=(e_k)_{k\in\N}$ of $V$ with bi-orthogonal sequence $e=(e_k^\ast)_{k\in\N}$. Here it is sufficient to
require that the considered stochastic processes have {\em $D$-cylindrical c{\`a}dl{\`a}g trajectories} for
$D=\{e_1^\ast,e_2^\ast,\dots\}$, that is
\begin{align*}
  \big(\scapro{X(t)}{e_k^\ast}\colon t\in [0,T]\big), \qquad
   \big(\scapro{X_n(t)}{e_k^\ast}\colon t\in [0,T]\big), \quad n\in\N,
\end{align*}
have c{\`a}dl{\`a}g trajectories  for all $k\in\N$. The notions of cylindrical c{\`a}dl{\`a}g 
and $D$-cylindrical c{\`a}dl{\`a}g paths  can be found in \cite{PesZab13}.

In order to have a clearer presentation of our paper, and not at least since our focus is rather on the different modes of convergence  instead of the subtle issue of temporal regularity, we require stochastic processes to have
c{\`a}dl{\`a}g trajectories in the underlying Banach space. However, if necessary, it should be obvious how to extend our results to stochastic processes with c{\`a}dl{\`a}g trajectories  only in the cylindrical sense.
\end{remark}

\section{Convergence of stochastic convolution integrals}\label{se.integral-convergence}

In this section we apply our results of Section \ref{se.conv-prob} to the convergence of stochastic convolution integrals with 
respect to L\'evy process. Although it would be possible to continue with the general setting of Banach 
spaces with a Schauder basis, we restrict ourselves here to Hilbert spaces in order to make use of 
standard integration theory as in \cite{Chojnowska87}. In this case, we identify the dual spaces $U^\ast$ 
and $V^\ast$ with the separable Hilbert spaces $U$ and $V$.

Let $\xi$ be an infinitely divisible Radon measure on $\Borel(U)$. Then the characteristic function of $\xi$ is given by
\begin{align*}
 \phi_{\xi}\colon U\to\C, \qquad \phi_{\xi}(u)=\exp\big(\Psi(u)\big),
\end{align*}
where the \emph{L{\'e}vy symbol} $\psi\colon U\to\C$ is defined by
\begin{align*}
 \Psi(u)= i\scapro{a}{u}-\tfrac{1}{2} \scapro{Qu}{u}
   +\int_U\left(e^{i\scapro{u}{r}}-1- i\scapro{u}{r}   \1_{B_U}(r)\right)\nu(dr)  ,
\end{align*}
where $a\in U$, $Q\colon U \to U$ is the covariance operator of a Gaussian Radon
measure on $\Borel(U)$ and $\nu$ is a $\sigma$-finite measure on $\Borel(U)$ with $\nu(\{0\})=0$
and
\begin{align*}
  \int_U \big(\norm{r}^2\wedge 1 \big)\,\nu(dr)<\infty.
\end{align*}
Consequently, the triplet $(a,Q,\nu)$ characterises the distribution of the Radon measure $\xi$ and thus,
it is  called it the \emph{characteristics of  $\xi$}. If $X$ is an $U$-valued random variable which is infinitely
divisible then we call the characteristics of its probability distribution the characteristics of $X$.
The L{\'e}vy symbol $\Psi\colon U^\ast\to \C$ is
sequentially weakly continuous and satisfies
\begin{align}
\label{eq.Levysymbolup}
\abs{\Psi(u)}\le c(1+\norm{u}^2)
\qquad\text{for all }u\in U,
\end{align}
for a constant $c>0$ depending on the underlying infinitely divisible distribution.

Let $\{\F_t\}_{t\ge 0}$ be a filtration for the probability space
$(\Omega,\A,P)$. An adapted stochastic process $L:=(L(t)\colon t\ge 0)$ with values in $U$ is called a
\emph{L{\'e}vy process} if $L(0)=0$ $P$-a.s., $L$ has independent and stationary increments and
$L$ is continuous in probability. It follows that there exists a
version of $L$ with paths which are continuous from the right and have limits from the left
(c{\`a}dl{\`a}g paths). In the sequel we always assume that a L{\'e}vy process has c{\`a}dl{\`a}g paths.
Clearly, the random variable $L(1)$ is infinitely divisible and we call its characteristics the characteristics of $L$.

In the work \cite{Chojnowska87}, Chojnowska--Michalik introduces a theory of stochastic integration for
deterministic, operator-valued integrands with respect to a $U$-valued L{\'e}vy process.
Another approach in a more general setting can be found in \cite{Riedle13} but we follow here
\cite{Chojnowska87}. Let $V$ be another separable Hilbert space and define
\begin{align*}
  \H^2(U,V):=\left\{F\colon [0,T]\to \L(U,V)\colon  F \text{ is measurable, } \int_0^T\norm{F(s)}_{U\to V}^2\,ds<\infty
  \right\}.
\end{align*}
For $F\in \H^2(U,V)$ we denote by $F^\ast(t)$ the adjoint operator $(F(t))^\ast\colon V\to U$ for each $t\in [0,T]$.
In \cite{Chojnowska87}, the author starts with step functions in $\H^2(U,V)$ to define a stochastic integral and finally shows, that
for each element in $\H^2(U,V)$ this stochastic integral exists as the limit of the stochastic integrals for
step functions in $\H^2(U,V)$ in the topology of convergence in probability. We denote this stochastic integral
for $F\in \H^2(U,V)$ with respect to the L{\'e}vy process $L$ by
\begin{align*}
  I(F):=\int_0^T F(s)\,dL(s).
\end{align*}
If $\Psi$ is the L{\'e}vy symbol of $L$ then the stochastic integral $I(F)$ is infinitely divisible and
has the characteristic function
\begin{align}
\label{eq.charint}
  \phi_{I(F)}\colon V\to\C, \qquad \phi_{I(F)}(v)= \exp\left(\int_0^T \Psi(F^\ast(s)v)\,ds\right).
\end{align}
By firstly considering step functions and then passing to the limit, one can show that for each
$F\in \H^2(U,V)$ the stochastic integral $I(F)$ obeys
\begin{align}
\label{eq.integralswap}
  \scapro{\int_0^T F(s)\,dL(s)}{v}  =\int_0^T F^\ast(s)v\,dL(s)
  \qquad\text{$P$-a.s. for all $v\in V$.}
\end{align}
Here, the right hand side is understood as the same stochastic integral but for
the integrand $F^\ast(\cdot)v\in \H^2(U,\R)$. If $F\in \H^2(U,V)$ is for some $v\in V$ of the special form
\begin{align*}
  F^\ast(t)v=\phi(t)Gv \qquad\text{for all }t\in [0,T],
\end{align*}
for a function $\phi\colon \R\to\R$ and $G\in \L(V,U)$ then one obtains
\begin{align}
\label{eq.realintegrands}
  \int_0^T F^\ast(s)v\, dL(s)=\int_0^T \phi(s)\,d\ell(s),
\end{align}
where $\ell$ denotes the real-valued L{\'e}vy process defined by
$\ell(t):=\scapro{L(t)}{G v}$. If $F\in \H^2(U,V)$ is of the special
form  $F(\cdot) = S(\cdot) G$ for some $G\in \L(U,V)$ and $S\in \H^2(V,V)$, we obtain
\begin{align}\label{eq.integral-imageLevy}
    \int_0^T G^\ast S^\ast(s)v\, dL(s)=\int_0^T S^\ast (s)v\,dK(s),
\end{align}
where $K$ is the L{\'e}vy process in $V$ defined by $K(t):=GL(t)$ for all $t\ge 0$.

For a function $F\in  \H^2(U,V)$ we define the \emph{stochastic convolution integral process} 
$F\ast L:=(F\ast L(t)\colon t\in [0,T])$ by
  \begin{align*}
    F\ast L(t):=\int_0^t F(t-s)\,dL(s)
    \qquad\text{for all }t\in [0,T].
  \end{align*}
In this section we apply our results of Section \ref{se.conv-prob} to the convergence of stochastic convolution integral
processes  in the weak  and product  topology $M_1$, that is for functions $F$, $F_\gamma\in \H^2(U,V)$,
depending on a parameter $\gamma>0$, we establish the convergence
\begin{align*}
  \lim_{\gamma\to\infty} F_\gamma \ast L=F\ast L
\end{align*}
in probability in the weak and product topology.

The study of the limiting behaviour requires that the stochastic processes have
c{\`a}dl{\`a}g paths in $V$, or at least in the appropriate cylindrical sense as pointed out in Remark
\ref{re.path-cylindrical}. There is no condition for regularities of trajectories available covering our 
rather general setting but for numerous specific situations one knows sufficient 
conditions guaranteeing either continuous or c{\`a}dl{\`a}g trajectories of stochastic convolution integrals. 
For example, classical results on continuity of
Gaussian processes can be found in \cite{MarcusShepp} and \cite{Talagrand87}, and on
regularity of infinitely divisible processes in \cite{Talagrand93}; temporal path regularity of stochastic convolution 
integrals are considered in \cite{KwapienMR-06}  and \cite{MarRos05},
the infinite-dimensional Ornstein-Uhlenbeck process is treated in  \cite{MilletSmolenski}
and \cite{PeszatSeidler}. As our work is focused on the convergence rather than
regularities of trajectories we will assume the following in this section: \\

{\bf Assumption A:} For all considered functions $F$,$F_\gamma\in \H^2(U,V)$,
$\gamma>0$, the stochastic processes $F\ast L$ and $F_\gamma\ast L$, $\gamma>0$,
have c{\`a}dl{\`a}g trajectories.

Furthermore, if $W$ denotes the Gaussian part of $L$ then the stochastic process
$F\ast W$ has continuous trajectories. \\

We do not assume that $F_\gamma\ast W$ has continuous paths but only the prospective limit $F\ast W$. This is a
quite natural assumption for the $M_1$ topology that only the  limit is continuous.

\subsection{Convergence of the marginals }

\begin{lemma}\label{le.fddRn}
Let $F$, $F_\gamma$, $\gamma>0$, be functions in $\H^2(U,V)$ satisfying for a subset $D\subseteq V$ and all $u\in U$
\begin{align}
{\rm (i)}
&\, F^\ast (\cdot)v,\,F_\gamma^\ast(\cdot)v\in D([0,T];U) \quad\text{for all }v\in D \text{ and }\gamma>0;\\
{\rm (ii)}\label{eq.th.fddRn-bound}
&\, \sup_{\gamma>0}\norm{\scapro{F_\gamma(\cdot)u}{v}}_\infty<\infty \quad\text{for all }v\in D ;\\
{\rm (iii)} &\,\text{for each $v\in D$ there exists a Lebesgue null set $B\in\Borel([0,T])$ such that}\notag \\
&\, \lim_{\gamma\to\infty} \scapro{\big(F_\gamma(s)-F(s)\big)u}{v}=0 \quad\text{for all }s\in B^c, \,u\in U. \label{eq.th.fddRn-conv}
\end{align}
Then  for each $t\in[0,T]$ and $v\in D$ we  have
\begin{align*}
\lim_{\gamma\to\infty} \scapro{F_\gamma\ast L(t)}{v}
= \scapro{F\ast L(t)}{v}
\end{align*}
in probability.
\end{lemma}
\begin{proof}
Define for each $t\in [0,T]$ and $\gamma>0$  the random variable
\begin{align*}
 X_\gamma(t):=\int_0^{t}\big(F_\gamma(t-s)-F(t-s)\big)\,
  dL(s).
\end{align*}
By linearity of the stochastic integral and since the Euclidean topology  in $\R^n$ coincides with the product topology it is sufficient to prove that for each $v\in D$ and $t\in [0,T]$ we have
\begin{align*}
\scapro{X_\gamma(t)}{v}\to 0
 \qquad\text{weakly in $\R$ as $\gamma\to\infty$.}
\end{align*}
Let $\Psi$ denote the L{\'e}vy symbol of $L$. Due to equality \eqref{eq.integralswap} we obtain for the  characteristic function of $X_\gamma$ for $\beta\in \R$ that
\begin{align}\label{eq.charXgamma}
E\left[\exp\big(i\beta \scapro{X_\gamma(t)}{v}\big)\right]
& =E\left[ \exp\left( i\beta \int_0^{t}\big(F^\ast_\gamma(t-s)-F^\ast(t-s)\big)v\,dL(s)\right)\right]\notag\\
& = \exp\left( \int_0^{t}\Psi\left(\big(F^\ast_\gamma(t-s)-F^\ast(t-s)\big)(\beta v)\right)\,ds\right)\notag\\
&= \exp\left( \int_0^{t}\Psi\left(\big(F^\ast_\gamma(s)-F^\ast(s)\big)(\beta v)\right)\,ds\right).
\end{align}
Fix $v\in D$ and define for each $\gamma>0$  the linear and continuous mapping
\begin{align*}
  T_\gamma\colon U\to D([0,T],\norm{\cdot}_\infty), \qquad T_\gamma u:=\scapro{u}{F_\gamma^\ast(\cdot)v}.
\end{align*}
Condition \eqref{eq.th.fddRn-bound} guarantees for each $u\in U$ that
 \begin{align*}
  \sup_{\gamma>0}\norm{T_\gamma u}_\infty=
   \sup_{\gamma>0}\norm{\scapro{u}{F_\gamma^\ast(\cdot)v}}_\infty<\infty.
 \end{align*}
Thus, the  uniform boundedness principle implies
\begin{align*}
M:=\sup_{\gamma>0}\sup_{s\in [0,T]}\norm{F^\ast_\gamma(s)v}_U =\sup_{\gamma>0}\sup_{s\in [0,T]}\sup_{\norm{u}\le 1}\abs{\scapro{u}{F_\gamma^\ast(s)v}}= \sup_{\gamma>0}\norm{T_\gamma}_{U\to D}<\infty.
\end{align*}
The estimate \eqref{eq.Levysymbolup} for the L{\'e}vy symbol $\Psi$ implies
that there exists a constant $c>0$ such that for each $\gamma>0$ and $s\in [0,T]$
\begin{align*}
\abs{\Psi\big( (F_\gamma^\ast(s)-F^\ast(s))(\beta v)\big)} &\le c\left(1 + \norm{ \big(F_\gamma^\ast(s)-F^\ast(s)\big)(\beta v)}^2\right)\\
&\le c\left(1 +2 \abs{\beta}^2\left(\norm{F_\gamma^\ast(s)v}^2 + \norm{F^\ast(s)v}^2\right) \right)\\
&\le c\left(1 +2 \abs{\beta}^2\left( M^2 + \norm{v}^2 \norm{F(s)}_{U\to V}^2\right)\right).
\end{align*}
Since Condition \eqref{eq.th.fddRn-conv} implies by the sequentially weak continuity of the L{\'e}vy symbol $\Psi\colon U^\ast\to \C$ that
\begin{align*}
\lim_{\gamma\to\infty}  \Psi\left( (F_\gamma^\ast(s)- F^\ast(s))(\beta v)\right)=0 \qquad\text{for Lebesgue almost  all }s\in [0,T],
\end{align*}
 Lebesgue's theorem of dominated convergence enables us to conclude
\begin{align*}
\lim_{\gamma\to\infty} \int_0^t\Psi\left((F_\gamma^\ast(s)-
F^\ast(s))(\beta v)\right)\,ds
=0,
\end{align*}
which completes the proof by \eqref{eq.charXgamma}.
\end{proof}

\subsection{The reproducing kernel Hilbert space}

In this subsection we fix a L{\'e}vy process $L$ in $U$ with characteristics $(a,Q,\nu)$ and let $\alpha$ be a positive constant. The L{\'e}vy process $L$ can be decomposed into
\begin{align}\label{eq.LevyIto}
   L(t)=W(t)+X_\alpha(t)+Y_\alpha(t)\qquad\text{for all $t\ge 0$,}
\end{align}
 where $W$ is a  Wiener process with covariance operator $Q\colon U\to U$, and $X_\alpha$ and $Y_\alpha$  are
 $U$-valued L{\'e}vy processes with characteristic functions
\begin{align*}
  \phi_{X_\alpha(t)}(u)&=\exp\left(-t\int_{\norm{r}\le \alpha} \left(e^{i\scapro{r}{u}}
  - 1 - i\scapro{r}{u}\right)\,\nu(dr)\right),\\
  \phi_{Y_\alpha(t)}(u)&=\exp\left(it\scapro{b_\alpha}{u}-t\int_{\alpha<\norm{r}} \left(e^{i\scapro{r}{u}}
  - 1 \right)\,\nu(dr)\right),
\end{align*}
for $u\in U$ and $t\ge 0$. The element $b_\alpha\in U$ is determined by the characteristics of $L$  and by the constant $\alpha$. Since
$X_\alpha(1)$ has finite moments we can define the covariance operator $R_\alpha\colon U\to U$ of $X_\alpha(1)$  by
\begin{align*}
  \scapro{R_\alpha u_1}{u_2}=E\big[ \scapro{X_\alpha(1)}{u_1}\scapro{X_\alpha(1)}{u_2}\big]
           \qquad\text{for all }u_1,u_2\in U.
\end{align*}
Since $R_\alpha$ is positive and symmetric there exists a separable Hilbert space $H_\alpha$ with norm
$\norm{\cdot}_{H_\alpha}$ and an embedding $i_\alpha\colon H\to U$ satisfying $R_\alpha=i_\alpha i_\alpha^\ast$.
In particular, we have
\begin{align}\label{eq.lim-i-alpha}
 \lim_{\alpha\to 0} \norm{i_\alpha}_{H_\alpha\to U}=0,
\end{align}
which follows from the estimate
\begin{align*}
\norm{i_\alpha^\ast u}^2_{H_\alpha}= \scapro{R_\alpha u}{u}=\int_{\norm{r}\le \alpha} \scapro{u}{r}^2\, \nu(dr)
           \le \norm{u}^2 \int_{\norm{r}\le \alpha} \norm{r}^2\, \nu(dr)
\end{align*}
for each $u\in U$.  One obtains for $\alpha \le \beta$
\begin{align*}
  \scapro{R_\alpha u}{u}
  = \int_{\norm{r}\le \alpha} \scapro{u}{r}^2\, \nu(dr)
   \le    \int_{\norm{r}\le \beta} \scapro{u}{r}^2\, \nu(dr)
  =  \scapro{R_\beta u}{u} \qquad\text{for all }u\in U.
\end{align*}
One can deduce  from Riesz representation theorem, see \cite[Proposition 1.1]{vanNeerven98} or
\cite[Section 1.1]{vanNeervenWeis}, that $H_\alpha\subseteq H_\beta$ and the embedding $H_\alpha \to H_\beta$ is contractive.

Since the range of $i_\alpha^\ast$ is dense in $H_\alpha$ and $H_\alpha$ is separable there exits a 
basis $(h_k^{\alpha})_{k\in \N}\subseteq i_\alpha^\ast(U)$. We choose $u_k^{\alpha}\in U$ such 
that $i^\ast_\alpha u_k^{\alpha}=h_k^{\alpha}$ 
for all $k\in \N$ and we define real-valued L{\'e}vy processes $\ell_k^{\alpha}$ by 
$\ell_k^{\alpha}(t):=\scapro{X_\alpha(t)}{u_k^{\alpha}}$ for all $t\in [0,T]$. The L{\'e}vy process $X_\alpha$ can be represented by
\begin{align}\label{eq.Levyassum}
  X_\alpha(t)=\sum_{k=1}^\infty i_\alpha h_k^{\alpha} \ell_k^{\alpha}(t)
  \qquad\text{for all }t\in [0,T],
\end{align}
where the sum converges weakly in  $L^2_P(\Omega;U)$,  i.e.
\begin{align*}
\scapro{X_\alpha(t)}{u}=\sum_{k=1}^\infty \scapro{i_\alpha h_k^{\alpha}}{u} \ell_k(t)
  \qquad\text{in $L^2_P(\Omega;\R)$ for all }t\in [0,T]\text{ and } u\in U.
\end{align*}
The representation \eqref{eq.Levyassum} is called {\em Karhunen--Lo{\`e}ve expansion}, and it can be
derived in the same way as it is done in Riedle \cite{Riedle10} for Wiener processes.
Note, that it follows easily from their definition that the L{\'e}vy processes $(\ell_k)_{k\in\N}$ are uncorrelated.

\subsection{Estimating the small jumps}

We begin with a generalisation to the Hilbert space setting of a result of Marcus and Rosi\'nski in
\cite{MarRos03} on a maximal inequality of convolution integrals. We apply here the decomposition  \eqref{eq.LevyIto} of the L{\'e}vy process $L$ for some $\alpha>0$.
\begin{lemma}\label{le.MarcusRosHilbert}
A function $f\in \H^2(U,\R)$ satisfies for each $\alpha>0$ the estimate
\begin{align*}
  E\left[\sup_{t\in [0,T]}\abs{\int_0^t f(t-s)\, dX_\alpha(s)}\right]
  \le  \varkappa \sqrt{2T}\sum_{k=1}^\infty \norm{\scapro{f(\cdot)}{i_\alpha h_k^\alpha}}_{TV_2} ,
\end{align*}
where $\varkappa:=32\sqrt{2} \int_0^1 \sqrt{\ln(1/s)}\,ds$.
\end{lemma}
\begin{proof}
Since $\alpha>0$ is fixed we neglect its notation  in the following. According to \eqref{eq.Levyassum} the L{\'e}vy process $X$ can be represented by
\begin{align*}
  X(t)=\sum_{k=1}^\infty ih_k \ell_k(t)
        \qquad\text{for all }t\in [0,T],
\end{align*}
where the sum converges weakly in $L^2_P(\Omega;U)$. Since the real-valued L{\'e}vy processes $(\ell_k)_{k\in\N}$ are uncorrelated we obtain
\begin{align*}
 E\left[\sup_{t\in [0,T]}\abs{\int_0^t f(t-s)\, dX(s)}\right]
&= E\left[\sup_{t\in [0,T]}\abs{\sum_{k=1}^\infty \int_0^t \scapro{f(t-s)}{ih_k}\, d\ell_k(s)}\right]\\
&\le \sum_{k=1}^\infty  E\left[\sup_{t\in [0,T]}\abs{\int_0^t \scapro{f(t-s)}{ih_k}\, d\ell_k(s)}\right].
\end{align*}
In order to estimate the expectation of the real-valued stochastic integrals on the right 
hand side we follow some arguments by Marcus and Rosi\'nski in 
\cite{MarRos05}. Let $\ell_k^\prime$ be an independent copy of $\ell_k$ and define the 
symmetrisation $\hat{\ell}_k:=\ell_k-\ell^\prime_k$ for each $k\in\N$. Since $E[\ell_k(t)]=0$ for all $t\ge 0$ and  $k\in\N$ one obtains
\begin{align*}
&E\left[\sup_{t\in [0,T]} \abs{\int_0^t \scapro{f(t-s)}{ih_k} \, d\ell_k(s)}\right]\\
&\qquad = E\left[\sup_{t\in [0,T]} \abs{E\left[ \int_0^t \scapro{f(t-s)}{ih_k} \, d\ell_k(s)
       -\int_0^t \scapro{f(t-s)}{ih_k} \, d\ell^\prime(s)\middle\vert \ell_k(T)\right]    }   \right]\\
&\qquad \le E\left[\sup_{t\in [0,T]}
       E\left[\abs{ \int_0^t \scapro{f(t-s)}{ih_k} \, d\hat{\ell}_k(s) }\middle\vert\ell_k(T)\right]\right]\\
&\qquad \le E\left[ E\left[\sup_{t\in [0,T]}
      \abs{ \int_0^t \scapro{f(t-s)}{ih_k} \, d\hat{\ell}_k(s) }\middle\vert \ell_k(T)\right]\right]\\
&\qquad =  E\left[\sup_{t\in [0,T]}\abs{ \int_0^t \scapro{f(t-s)}{ih_k} \, d\hat{\ell}_k(s) }\right].
\end{align*}
Theorem 1.1 in \cite{MarRos03} guarantees for all $k\in \N$ that
\begin{align*}
  E\left[\sup_{t\in [0,T]}\abs{\int_0^t \scapro{f(t-s)}{ih_k}\, d\hat{\ell}_k(s)}\right]
\le \varkappa E\left[\abs{\int_0^T \norm{\scapro{f(\cdot)}{ih_k}\1_{[0,T-s]}}_{TV_2}\, d\hat{\ell}_k(s)}\right].
\end{align*}
Since the L{\'e}vy measure $\lambda_k$ of $\ell_k$ is given by $\lambda_k:=\nu_\alpha \circ \scapro{\cdot}{u_k}^{-1}$
where $\nu_\alpha$ denotes the L{\'e}vy measure of $X$ we conclude
\begin{equation}
\label{eq.estmatingTV}
\begin{aligned}
&\left(E \abs{\int_0^T \norm{\scapro{f(\cdot)}{ih_k}\1_{[0,T-s]}}_{TV_2}\,
     d\hat{\ell}_k(s)}\right)^2\\
&\qquad\qquad\le E \abs{\int_0^T \norm{\scapro{f(\cdot)}{ih_k}\1_{[0,T-s]}}_{TV_2}\,
     d\hat{\ell}_k(s)}^2\\
&\qquad\qquad=2 \int_0^T \int_{\R} \norm{\scapro{f(\cdot)}{ih_k}\1_{[0,T-s]}}_{TV_2}^2 \beta^2
                \,\lambda_k(d\beta) \,ds\\
&\qquad\qquad =2\int_0^T \norm{\scapro{f(\cdot)}{ih_k}\1_{[0,T-s]}}_{TV_2}^2\,ds
                     \int_{U} \scapro{u}{u_k}^2\,\nu_\alpha(du)\\
&\qquad\qquad =2\int_0^T \norm{\scapro{f(\cdot)}{ih_k}\1_{[0,T-s]}}_{TV_2}^2\,ds \,\scapro{Ru_k}{u_k} \\
&\qquad\qquad =2\int_0^T \norm{\scapro{f(\cdot)}{ih_k}\1_{[0,T-s]}}_{TV_2}^2\,ds \\
&\qquad\qquad \le 2T\norm{\scapro{f(\cdot)}{ih_k}}_{TV_2}^2,
\end{aligned}
\end{equation}
which completes the proof.
\end{proof}

The bound on the right hand side in Lemma \ref{le.MarcusRosHilbert} depends on the regularity of the function $f$
and of the covariance structure of the underlying L{\'e}vy process. It is a natural generalisation to the
infinite dimensional setting of the result in \cite{MarRos03}.

We will later consider the special case, that the integrands of the stochastic convolution 
integrals can be diagonalised  with respect to an orthonormal basis. In this case, we can improve the analogue estimate of Lemma \ref{le.MarcusRosHilbert}.
\begin{lemma}\label{le.MarcusRosHilbert-realvalued}
Let $V$ be a Hilbert space with an orthonormal basis $(e_k)_{k\in\N}$ and let $F\in\H^2(U,V)$ be a function  of the form
\begin{align*}
  F^\ast (\cdot)e_k =\phi_k(\cdot)\, G e _k \qquad \text{for all }k\in\N,
\end{align*}
for  c{\`a}dl{\`a}g functions $\phi_k\colon [0,T]\to \R$ and $G\in \L(V,U)$. Then it follows for each $\alpha>0$ and $v\in V$ that
\begin{align*}
&  E\left[\sup_{t\in [0,T]}\abs{\int_0^t F^\ast(t-s)v\, dX_\alpha(s)}\right]\\
&\qquad\qquad  \le    \varkappa \sqrt{2T}\sum_{k=1}^\infty
        \abs{\scapro{v}{e_k}}\norm{\phi_k(\cdot)}_{TV_2}\left(\int_{\norm{u}\le \alpha} \abs{\scapro{u}{Ge_k}}^2\,\nu(du)\right)^{1/2},
\end{align*}
where $\varkappa:=32\sqrt{2} \int_0^1 \sqrt{\ln(1/s)}\,ds$.
\end{lemma}
\begin{proof}
For each $k\in\N$ define the real-valued L{\'e}vy process $x^{\alpha}_k$ by defining $x^{\alpha}_k(t)=\scapro{X_\alpha(t)}{Ge_k}$ for all $t\ge 0$. It follows from \eqref{eq.realintegrands} for each $v\in V$ that
\begin{align*}
 E\left[\sup_{t\in [0,T]}\abs{\int_0^t F^\ast(t-s)v \, dX_\alpha(s)}\right]
&= E\left[\sup_{t\in [0,T]}\abs{\sum_{k=1}^\infty\scapro{v}{e_k}
       \int_0^t \phi_k(t-s) Ge_k \, dX_\alpha(s)}\right]\\
&= E\left[\sup_{t\in [0,T]}\abs{\sum_{k=1}^\infty \scapro{v}{e_k}
       \int_0^t \phi_k(t-s)  \, dx^\alpha_k(s)}\right]\\
&\le \sum_{k=1}^\infty \abs{\scapro{v}{e_k}} E\left[\sup_{t\in [0,T]}
       \abs{\int_0^t \phi_k(t-s) \, dx^\alpha_k(s)}\right].
\end{align*}
As in the proof of Lemma \ref{le.MarcusRosHilbert} we obtain
\begin{align*}
  E\left[\sup_{t\in [0,T]}\abs{\int_0^t \phi_k(t-s)\,dx^\alpha_k(s)}\right]
\le \varkappa E\left[\abs{\int_0^T \norm{\phi_k(\cdot)\1_{[0,T-s]}}_{TV_2}\,
              d\hat{x}^\alpha_k(s)}\right],
\end{align*}
where $\hat{x}_k^\alpha:=x_k^\alpha-x^{\alpha\prime}_k$ denotes the symmetrisation  of $x_k^\alpha$ for 
each $k\in\N$ by an independent copy $x^{\alpha\prime}_k$ of $x^\alpha_k$. 
Since the L{\'e}vy measure $\lambda_k^\alpha$ of $x^\alpha_k$ is given by
$\lambda^\alpha_k=\nu_\alpha\circ \scapro{\cdot}{Ge_k}^{-1}$ where $\nu_\alpha$ denotes the L{\'e}vy measure of 
$X_\alpha$  we conclude by a similar calculation as in \eqref{eq.estmatingTV} that
\begin{align*}
\left(E\left[\abs{\int_0^T \norm{\phi_k(\cdot)\1_{[0,T-s]}}_{TV_2}\,
d\hat{x}^\alpha_k(s)}\right]\right)^2
\le 2T\norm{\phi_k(\cdot)}_{TV_2}^2 \int_{\norm{u}\le \alpha}
\scapro{u}{Ge_k}^2\,\nu(du).
\end{align*}
Summarising the estimates above completes the proof.
\end{proof}

\subsection{The general case}

In this section we consider the convergence in the weak topology and the product topology. For the latter
we assume that $e=(e_k)_{k\in\N}$ is an orthonormal basis of $V$.
\begin{theorem}\label{th.Dconv}
Let $F, F_\gamma\in\H^2(U,V)$, $\gamma>0$,  be functions
satisfying for a subset $D\subseteq V$
\begin{align}
 {\rm (i)}&\, F^\ast(\cdot)v,\, F_\gamma^\ast (\cdot)v\in D([0,T];U) \quad\text{for all }v\in D
  \text{ and } \gamma>0;\label{eq.proMconv-cadlag}\\
  {\rm (ii)}&\, \sup_{\gamma>0}\norm{\scapro{F_\gamma(\cdot)u}{v}}_\infty<\infty
  \quad\text{for all }u\in U,\, v\in D ; \label{eq.proMconv-bound}\\
    {\rm (iii)}&\, \limsup_{\alpha\to 0}\sup_{\gamma> 0}\sum_{k=1}^\infty \norm{\scapro{F_\gamma^\ast(\cdot)v}{i_\alpha h_k^\alpha}}_{TV_2}=0 \quad\text{for all }v\in D;\label{eq.proMconv-TV}\\
  {\rm (iv)}&\, \lim_{\gamma\to\infty}d_M\big(\scapro{F_\gamma(\cdot)u}{v},\,\scapro{F(\cdot)u}{v}\big)=0
   \quad\text{for all }u\in U,\, v\in D . \label{eq.proMconv-conv}
\end{align}
\begin{enumerate}
  \item[{\rm (1)}]
If $\{e_1,e_2,\dots\}\subseteq D$ then it follows
\begin{align*}
\lim_{\gamma\to\infty} \big(F_\gamma \ast L(t)\colon t\in [0,T]\big)
= \big(F \ast L(t)\colon t\in [0,T]\big)
\end{align*}
in probability in the product topology $\big(D([0,T];V), d_M^{{\,e}} \big)$.
\item[{\rm (2)}]
If $V=D$ then it follows
\begin{align*}
\lim_{\gamma\to\infty} \big(F_\gamma \ast L(t)\colon t\in [0,T]\big)
= \big(F \ast L(t)\colon t\in [0,T]\big)
\end{align*}
weakly in probability in $D([0,T];V)$.
\end{enumerate}
\end{theorem}
\begin{proof}
We show that each $v\in D$ the stochastic processes $X$ and $X_\gamma$, $\gamma>0$,
defined by
\begin{align*}
  X(t):=\scapro{\int_0^t F(t -s)\,dL(s)}{v},
  \quad X_\gamma(t):=\scapro{\int_0^t F_\gamma(t -s)\,dL(s)}{v},
  \quad\text{ $t\in [0,T]$,}
\end{align*}
 satisfy the  conditions in Theorem \ref{th.stochconv=fdd+M}. Note, that
 for $0\le t_1\le t_2\le T$ and $\beta\in\R$ we have
\begin{align*}
&E\left[\exp\left(i\beta (X(t_1)-X(t_2))\right)\right]\\
&\quad =\exp\left(\int_0^{t_1} \Psi\big((F^\ast(t_1-s)-F^\ast(t_2-s))(\beta v)\big)\,ds\right)
\exp\left(\int_{t_1}^{t_2} \Psi\big((F^\ast(t_2-s)(\beta v)\big)\,ds\right).
\end{align*}
Since $F^\ast(\cdot)(\beta v)$ is Lebesgue almost everywhere continuous due to \eqref{eq.proMconv-cadlag} it follows
that the stochastic process $X$ and analogously $X_\gamma$ are stochastically continuous.

It follows from Condition \eqref{eq.proMconv-conv} by Theorem 12.5.1 in \cite{whitt02}
for each $u\in U$ and $v\in D$  that
\begin{align}\label{eq.F-weakly-cont}
  \lim_{\gamma\to\infty}\scapro{F_\gamma(t)u}{v}=\scapro{F(t)u}{v}
  \qquad\text{for all }t \in \big(J( \scapro{F(\cdot)u}{v})\big)^c.
\end{align}
Since the set $J( F^\ast(\cdot)v)$ of discontinuities of $F^\ast(\cdot)v$ is a Lebesgue null set by Condition
\eqref{eq.proMconv-cadlag} and satisfies
$\big(J( \scapro{F(\cdot)u}{v})\big)\subseteq \big(J( F^\ast(\cdot)v)\big) $ for every $u\in U$, Lemma \ref{le.fddRn} guarantees that Condition (i) in Theorem \ref{th.stochconv=fdd+M} is satisfied.

In order to show Condition (ii) we have to establish for every $\epsilon>0$ and $v\in D$ that
\begin{align}\label{eq.pprove-cond(ii)}
  \lim_{\delta\searrow 0}\limsup_{\gamma\to\infty} P\left(M\left(\scapro{\int_0^\cdot F_\gamma(\cdot -s)\,dL(s)}{v},\,\delta\right)\ge \epsilon\right)=0.
\end{align}
For this purpose, fix some constants $\epsilon_1,\epsilon_2>0$ and $v\in D$.
Condition \eqref{eq.proMconv-TV} enables us to choose a constant $\alpha>0$  such that
\begin{align}\label{eq.constantalpha}
\sup_{\gamma>0}\sum_{k=1}^\infty \norm{\scapro{F_\gamma^\ast (\cdot)v}{i_\alpha h_k^\alpha}}_{TV_2}
\le \frac{\epsilon_1\epsilon_2}{2\varkappa\sqrt{2T} }
\end{align}
for $\varkappa:=32\sqrt{2} \int_0^1 \sqrt{\ln(1/s)}\,ds$.
As in \eqref{eq.LevyIto} we decompose the L{\'e}vy process $L$ into $L(t)=W(t)+X_\alpha(t)+Y_\alpha(t)$ for all $t\ge 0$, 
but we suppress the notion of $\alpha$ in the sequel. Here, $W$ is a  Wiener process with covariance operator 
$Q\colon U\to U$, and $X$ and $Y$  are $U$-valued L{\'e}vy processes with characteristic functions
\begin{align*}
  \phi_{X(t)}(u)&=\exp\left(-t\int_{\norm{r}\le \alpha} \left(e^{i\scapro{r}{u}}
  - 1 - i\scapro{r}{u}\right)\,\nu(dr)\right),\\
  \phi_{Y(t)}(u)&=\exp\left(it\scapro{b}{u}-t\int_{\alpha<\norm{r}} \left(e^{i\scapro{r}{u}}
  - 1 \right)\,\nu(dr)\right),
\end{align*}
for all $t\ge 0$, $u\in U$. 
It follows from \eqref{eq.integralswap} for every $t\in [0,T]$ and $\gamma>0$ that
\begin{align*}
\scapro{\int_0^t F_\gamma(t-s)\,dL(s)}{v}
&=\int_0^t F_\gamma^\ast (t-s)v \, dL(s).
\end{align*}
By the decomposition of $L$ we obtain the  representation
\begin{align*}
I_{\gamma}(t)&:=\int_0^t F_\gamma^\ast(t-s)v\,dL(s) = C_{\gamma}(t)+ A_{\gamma}(t)+ B_{\gamma}(t),
\intertext{where we define the real-valued random variables}
 C_{\gamma}(t)&:=\int_0^t F_\gamma^\ast(t-s)v\,dW(s),\\
  A_{\gamma}(t)&:= \int_0^t F_\gamma^\ast(t-s)v\,dY(s) ,\quad
 B_{\gamma}(t):= \int_0^t F_\gamma^\ast (t-s)v\,dX(s).
\end{align*}
In the sequel, we will consider the three stochastic integrals separately.

\noindent
1) We show that
\begin{align}\label{eq.pWestimate}
  \lim_{\delta\searrow 0}\limsup_{\gamma\to\infty} P\left(\sup_{\substack{t_1,t_2\in [0,T] \\ \abs{t_2-t_1}\le \delta}} \abs{C_{\gamma}(t_2)-C_{\gamma}(t_1)}\ge \epsilon_1\right)=0.
\end{align}
Let the covariance operator $Q$ of $W$ be decomposed into $Q=i_Qi^\ast_Q$ for some $i_Q\in \L(H_Q,U)$ and for a Hilbert space $H_Q$. Since the characteristic function $\phi_{W(1)}$ of
$W(1)$ is sequentially weakly continuous and is given by
\begin{align*}
  \phi_{W(1)}(u)=\exp\left(-\tfrac{1}{2}\norm{i^\ast u}_{H_Q}^2\right)
  \qquad\text{for all } u\in U,
\end{align*}
also the function $i^\ast\colon U\to H_Q$ is sequentially weakly continuous. Consequently,  
we can conclude from \eqref{eq.F-weakly-cont} that
\begin{align*}
  \lim_{\gamma\to \infty}\norm{i_Q^\ast\big(F^\ast_\gamma (s)v-F^\ast(s)v)}^2_{H_Q}=0
  \qquad\text{for Lebesgue-a.a.\ }s\in [0,T].
\end{align*}
Due to \eqref{eq.proMconv-bound} Lebesgue's theorem of dominated convergence
implies
\begin{align}\label{eq.intF*W}
\lim_{\gamma\to \infty}  \int_0^T \norm{i_Q^\ast\big(F^\ast_\gamma (s)v-F^\ast(s)v)}^2_{H_Q}\,ds=0.
\end{align}
Let $I:=[0,T]\cap \Q$ and define for each $\gamma>0$ the stochastic processes
$G_\gamma:=(G_\gamma(t)\colon  t\in I)$ where $G_\gamma(t):=C_\gamma(t)-\scapro{F\ast W(t)}{v}$.
Since $G_\gamma$ has independent increments it is a real-valued Gaussian process with
a.s.\ bounded sample paths due to Assumption A. By denoting $ M_\gamma:=\sup_{t\in I}G_\gamma(t)$
it follows for each $\delta>0$ that
\begin{align}\label{eq.C-G-M}
& \sup_{\substack{t_1,t_2\in I \\ \abs{t_2-t_1}\le \delta}} \abs{C_{\gamma}(t_2)-C_{\gamma}(t_1)}\notag\\
&\qquad \le 2\sup_{t\in I}\abs{G_\gamma(t)-E[M_\gamma]}+
\sup_{\substack{t_1,t_2\in I \\ \abs{t_2-t_1}\le \delta}}\abs{\scapro{F\ast W(t_2)-F\ast W(t_1)}{v}}.
\end{align}
Borell's inequality, see \cite[Theorem 2.1]{Adler90}, (or Borell--Tsirelson--Ibragimov--Sudakov inequality)
implies $E[M_\gamma]<\infty$  and that for every $\epsilon>0$
\begin{align*}
  P\left(\sup_{t\in I}\abs{G_\gamma(t)-E[M_\gamma]}\ge \epsilon\right)
  \le 2  P\left(\sup_{t\in I}G_\gamma(t)-E[M_\gamma]\ge \epsilon\right)
  \le \exp\left( -\frac{\epsilon^2}{2\sigma_\gamma^2}\right),
\end{align*}
where $\sigma_\gamma^2:=\sup_{t\in I} E[G_\gamma(t)^2]$.
Since equation \eqref{eq.intF*W} guarantees
\begin{align*}
 \sigma_\gamma^2 &=\int_0^T \norm{i_Q^\ast \left(F_\gamma^\ast(s)v-F^\ast(s)v\right)}_{H_Q}^2\,ds
 \to 0\quad\text{as }\gamma\to\infty,
\end{align*}
we can conclude
\begin{align}\label{eq.C-G-M-M}
  \lim_{\gamma\to\infty}  P\left(\sup_{t\in I}\abs{G_\gamma(t)-E[M_\gamma]}\ge \frac{\epsilon_1}{4}\right)
  =0.
\end{align}
As  the stochastic process $(\scapro{F\ast W(t)}{v}\colon t\in [0,T])$ is continuous according to Assumption A it follows
\begin{align}\label{eq.C-G-M-ast}
  \lim_{\delta\searrow 0}\sup_{\substack{t_1,t_2\in I \\ \abs{t_2-t_1}\le \delta}}\abs{\scapro{F\ast W(t_2)-F\ast W(t_1)}{v}}
  =0 \quad\text{$P$-a.s.}
\end{align}
Equations \eqref{eq.C-G-M-M} and \eqref{eq.C-G-M-ast}
establish \eqref{eq.pWestimate}.

\noindent
2) Lemma \ref{le.MarcusRosHilbert} implies by Markov inequality for each $\gamma>0$
\begin{align}\label{eq.pXestimate}
P\left(\sup_{t\in [0,T]} \abs{B_\gamma(t)}\ge \epsilon_1\right)
\le \frac{\varkappa \sqrt{2T}}{\epsilon_1}\sum_{k=1}^\infty \norm{\scapro{F_\gamma^\ast (\cdot)v}{i_\alpha h_k^\alpha}}_{TV_2}
\le \frac{\epsilon_2}{2}.
\end{align}

\noindent
3) Let $(N(t)\colon t\in [0,T])$ denote the counting process for $Y$, i.e.
\begin{align*}
  N(t):=\sum_{s\in [0,t]} \1_{\{\norm{\Delta Y(s)}> \alpha\} }
  \qquad\text{for all }t\in [0,T],
\end{align*}
 and let $\tau_{j}$, $j\in \N\cup\{0\}$, denote the jump times of $Y$,
 recursively defined by $\tau_{0}=0$ and
 \begin{align*}
   \tau_{j}:=\inf\big\{t> \tau_{j-1}\colon \norm{\Delta Y(t)}>\alpha\big\}
   \qquad\text{for all }j\in \N
 \end{align*}
with $\inf \emptyset =\infty$.
Note, that the jump times of $Y$ are countable in increasing order since the jump size of $Y$ is bounded from below by $\alpha$.
Since $Y$ is a pure jump process with  drift $b$, we obtain
\begin{align*}
&A_\gamma(t)=\int_0^t F_\gamma^\ast(t-s)v\, dY(s)
  = R_\gamma (t) + S_\gamma(t),
\intertext{where we define for every $t\in[0,T]$ and $\gamma>0$ }
& R_\gamma (t):= \sum_{j=0}^{N(t)}R_\gamma^j(t),\qquad
R_\gamma^j(t):=\scapro{F_\gamma^\ast(t-\tau_{j})v}{\Delta Y(\tau_{j})}\1_{[\tau_{j},\infty)}(t),\\
& S_\gamma(t):=  \int_0^t \scapro{F_\gamma^\ast(t-s)v}{b}\,ds.
\end{align*}
Analogously, we have
\begin{align*}
&A(t)=\int_0^t F^\ast(t-s)v\, dY(s)
  = R (t) + S(t),
\intertext{where we define for every $t\in[0,T]$  }
& R(t):= \sum_{j=0}^{N(t)}R^j(t),\qquad
R^j(t):=\scapro{F^\ast(t-\tau_{j})v}{\Delta Y(\tau_{j})}\1_{[\tau_{j},\infty)}(t),\\
& S(t):=  \int_0^t \scapro{F^\ast(t-s)v}{b}\,ds.
\end{align*}
Due to Condition \eqref{eq.proMconv-cadlag}
the stochastic process $R^j:=(R^j(t)\colon t\in [0,T])$
 has c{\`a}dl{\`a}g paths for each $j\in \N\cup\{0\}$ and the random set $J(R^j)$ of its discontinuities  satisfies
\begin{align*}
  J(R^j)&\subseteq \Big\{t\in [0,T]\colon  t-\tau_j\in J\big(F^\ast(\cdot)v\big),\, t\in (\tau_j,T]\Big\}\cup\{\tau_j\}.
\end{align*}
Consequently, the set of joint discontinuities of
$R^i$ and $R^j$ for $i<j$ satisfies
\begin{align*}
  J(R^i)\cap J(R^j)
  \subseteq \Big\{t\in [0,T]\colon  t=\tau_j(\omega)-\tau_i(\omega)
 \in J\big(F^\ast(\cdot)v\big)-J\big(F^\ast(\cdot)v\big)\,\text{ for some }\omega\in\Omega\Big\},
\end{align*}
where we apply the convention $\infty-\infty=\infty$.
Since the deterministic set
\begin{align*}
  J\big(F^\ast(\cdot)v\big)-J\big(F^\ast(\cdot)v\big)=\{t\in [0,T]\colon 
   t=s_1-s_2\text{ for some } s_1,s_2\in J\big(F^\ast(\cdot)v\big)\}
\end{align*}
is at most countable and the random vector $(\tau_i,\tau_j)$
is absolutely continuous for every $i< j$ it follows that
\begin{align*}
P\big( \text{$R^i$ and $R^j$ have joint discontinuities} \big)
&= P\big(J(R^i)\cap J(R^j)\big)\\
&\le P\big( \tau_j-\tau_i \in J\big(F^\ast(\cdot)v\big) - J\big(F^\ast(\cdot)v\big) \big)
=0.
\end{align*}
Consequently, for  the stochastic processes $(R^j)_{j\in\N}$ restricted to the complement $S^c$ of the null set
\begin{align*}
  S:=\bigcup_{i=2}^\infty \bigcup_{j=1}^{i-1} 
  \big\{\omega\in\Omega\colon \tau_i(\omega)-\tau_j(\omega) \in J\big(F^\ast(\cdot)v\big)\cap J\big(F^\ast(\cdot)v\big)\big\},
\end{align*}
there does not exist any $i\neq j$ such that $R^i$ and $R^j$ have joint discontinuities. Since Condition \eqref{eq.proMconv-conv} guarantees that for each $\omega\in \Omega $
\begin{align*}
\scapro{F_\gamma^\ast(\cdot-\tau_{j}(\omega))v}{\Delta Y(\tau_{j})(\omega)}
\to \scapro{F^\ast(\cdot-\tau_{j}(\omega))v}{\Delta Y(\tau_{j})(\omega)}
\end{align*}
in $(D([0,T];\R), d_M)$ it follows that
$R_\gamma^j(\omega)\to R^j(\omega)$  in  $(D([0,T];\R),d_M)$ for all  $\omega\in\Omega$ and $j\in\N$. Due to the 
disjoint sets of discontinuities on $S^c$, Corollary 12.7.1 in \cite{whitt02} guarantees
that $R_\gamma(\omega)$ converges to $R(\omega)$ in $(D([0,T];\R),d_M)$ for all $\omega\in S^c$. The deterministic analogue of
Theorem \ref{th.stochconv=fdd+M}, i.e.\ \cite[Theorem 12.5.1]{whitt02}, implies that
\begin{align*}
  \lim_{\delta\searrow 0}\limsup_{\gamma\to\infty} M(R_\gamma(\omega);\delta)=0
  \qquad\text{for all $\omega\in S^c$.}
\end{align*}
 By combining with Fatou's Lemma it follows that there exists $\delta_0>0$ such that for all $\delta\le \delta_0$ we have
\begin{align*}
\limsup_{\gamma\to\infty} P\Big( M(R_\gamma;\delta)\ge \tfrac{1}{2}\epsilon_1\Big)
&\le  P\Big(\limsup_{\gamma\to\infty} \{M(R_\gamma;\delta)\ge \tfrac{1}{2}\epsilon_1\}\Big)
\\
&= P\Big(\limsup_{\gamma\to\infty}   M(R_\gamma;\delta)\ge \tfrac{1}{2}\epsilon_1\Big) 
\le \epsilon_2.
\end{align*}
Since inequality \eqref{eq.f+gcont} implies
\begin{align*}
M(A_\gamma;\delta)=  M(R_\gamma + S_\gamma;\delta)
  &\le M(R_\gamma;\delta)+ \sup_{\substack{s_1,s_2\in [0,T]\\ \abs{s_2-s_1}\le \delta}}\abs{S_\gamma(s_1)-S_\gamma(s_2)}\\
  &\le  M(R_\gamma;\delta)+ \delta \sup_{\gamma>0}\norm{\scapro{v}{F_\gamma (\cdot)b}}_\infty,
\end{align*}
it follows that for all $\delta\le \delta_1$ where  
$\delta_1:=\min\{\delta_0,\, \tfrac{\epsilon_1}{3} (\sup_{\gamma>0}\norm{\scapro{v}{F_\gamma (\cdot)b}}_\infty)^{-1}\}$ we have
\begin{align}\label{eq.pYestimate}
\limsup_{\gamma\to\infty} P\Big( M(A_\gamma;\delta)\ge\epsilon_1\Big)
\le \limsup_{\gamma\to\infty}
P\Big( M(R_\gamma;\delta)\ge  \tfrac{1}{2}\epsilon_1\Big)
\le \epsilon_2.
\end{align}

\noindent
4) It follows from \eqref{eq.pWestimate} and\eqref{eq.pXestimate}  that
there exists  $\delta_2>0$ such that for each $\gamma>0$ the set
\begin{align*}
  E_\gamma:=\left\{\sup_{\abs{t_2-t_1}\le \delta_2} \abs{C_\gamma(t_2)-C_\gamma(t_1)}\le \epsilon_1\right\}
  \cap \left\{ \sup_{t\in [0,T]}\abs{B_\gamma(t)}\le \epsilon_1\right\}
\end{align*}
satisfies $P(E_\gamma)\ge 1-\epsilon_2$.
It follows from \eqref{eq.pYestimate}
by the inequalities \eqref{eq.f+ginfty} and \eqref{eq.f+gcont} for all
$\delta\le \min\{\delta_1,\delta_2\}$ that
\begin{align*}
\limsup_{\gamma\to\infty} P\Big(  M(I_\gamma;\delta))\ge 4\epsilon_1\Big)
\le \limsup_{\gamma\to\infty}  P\big( M(A_\gamma;\delta)\ge \epsilon_1\big)+  \limsup_{\gamma\to\infty} P(E^c_\gamma)
\le 2\epsilon_2,
\end{align*}
which shows \eqref{eq.pprove-cond(ii)} and thus, Condition (ii) in Theorem \ref{th.stochconv=fdd+M}.
\end{proof}

\subsection{The diagonal case}

In this section we consider the special case that the integrands can be diagonalised. For that purpose,
assume that $(e_k)_{k\in\N}$ is an orthonormal basis of $V$.
\begin{corollary}\label{co.Mproductconv-diagonal}
Let $F, F_\gamma\in\H^2(U,V)$, $\gamma>0$,  be functions of the form
\begin{align*}
  F^\ast (s)e_k =\phi^k(s)\, G e _k,\qquad\qquad  F_\gamma^\ast (s)e_k =\phi_\gamma^k(s)\, G e_k \qquad\text{for all }s\in [0,T], \,k\in\N, \end{align*}
for c{\`a}dl{\`a}g functions $\phi^k$, $\phi_\gamma^k\colon [0,T]\to \R$ and $G\in \L(V,U)$. If
\begin{align}
  {\rm (i)}&\, \sup_{\gamma>0} \norm{\phi_\gamma^k}_\infty<\infty\qquad\text{for all }k\in\N;\label{eq.thmMconvprod-diagonal-bound}\\
    {\rm (ii)}&\, \sup_{\gamma> 0} \norm{\phi_\gamma^k}_{TV_2} <\infty\qquad\text{for all }k\in\N;\label{eq.thmMconvprod-diagonal-TV}\\
  {\rm (iii)}&\, \lim_{\gamma\to\infty}\phi_\gamma^k = \phi^k
  \text{in $\big(D([0,T];\R),d_M\big)$}\qquad\text{for all }k\in\N.\label{eq.thmMconvprod-diagonal-phiinM}
\end{align}
then it follows for each $\epsilon>0$ and $k\in \N$ that
\begin{align*}
\lim_{\gamma\to\infty} \big(F_\gamma \ast L(t)\colon t\in [0,T]\big)
= \big(F \ast L(t)\colon t\in [0,T]\big)
\end{align*}
 in probability in the product topology $\big(D([0,T];V), d_M^{\,e}\big)$.
\end{corollary}
\begin{proof}
The proof is analogously to the proof of Theorem \ref{th.Dconv} for $D=\{e_1,e_2,\dots\}$ but 
only the estimate \eqref{eq.pXestimate} is derived in the following way: for each $k\in\N$ and $\gamma>0$  Lemma
\ref{le.MarcusRosHilbert-realvalued} implies
\begin{align*}
E\left[\sup_{t\in [0,T]} \abs{\int_0^t F_\gamma^\ast(t-s)e_k\,dX(s) } \right]
\le \varkappa \sqrt{2T} \norm{\phi_k(\cdot)}_{TV_2}\left(\int_{\norm{u}\le \alpha} \abs{\scapro{u}{Ge_k}}^2\,\nu(du)\right)^{1/2},
\end{align*}
where $\varkappa:=32\sqrt{2} \int_0^1 \sqrt{\ln(1/s)}\,ds$ and $\alpha>0$ denotes the bound of the truncation function.  Condition \eqref{eq.thmMconvprod-diagonal-TV} enables us to choose for every $\epsilon_1$, $\epsilon_2>0$ and $k\in\N$ the constant  $\alpha>0$ small enough such that
\begin{align*}
 P\left(\sup_{t\in [0,T]} \abs{\int_0^t F_\gamma^\ast(t-s)e_k\,dX(s)}\ge \epsilon_1\right)
 \le \frac{\epsilon_2}{2}
 \qquad\text{for all }\gamma>0,
\end{align*}
which corresponds to inequality \eqref{eq.pXestimate}. Now we can follow the proof of Theorem \ref{th.Dconv}
once the Conditions \eqref{eq.proMconv-cadlag}, \eqref{eq.proMconv-bound} and \eqref{eq.proMconv-conv} are established. 
The first two conditions follow directly from the c{\`a}dl{\`a}g property of $\phi^k$, $\phi_\gamma^k$ and 
\eqref{eq.thmMconvprod-diagonal-bound}. Condition  \eqref{eq.proMconv-conv} is satisfied since 
Condition \eqref{eq.thmMconvprod-diagonal-phiinM} implies by Theorem 12.7.2 in \cite{whitt02} for each $k\in \N$ and $u\in U$ that
\begin{align*}
 \lim_{\gamma\to\infty} \scapro{F_{\gamma}(\cdot)u}{e_k}
  = \lim_{\gamma\to\infty}\phi_\gamma^k(\cdot) \scapro{u}{Ge_k}
  = \phi^k(\cdot) \scapro{u}{Ge_k}
  = \scapro{F(\cdot)u}{e_k},
\end{align*}
in $\big(D([0,T];\R), d_M\big)$.
\end{proof}

\begin{corollary}\label{co.M-weak-diagonal-conv}
Let  $F, F_\gamma\in\H^2(U,V)$, $\gamma>0$,  be functions of the form
\begin{align*}
  F^\ast (s)e_k =\phi^k(s)\, G e_k,\qquad\qquad  
  F_\gamma^\ast (s)e_k =\phi_\gamma^k(s)\, G e_k \qquad\text{for all }s\in [0,T], \,k\in\N, 
\end{align*}
for c{\`a}dl{\`a}g functions $\phi^k$, $\phi_\gamma^k\colon [0,T]\to \R$ and $G\in \L(V,U)$. If
\begin{align}
 {\rm (i)}&\,
  F^\ast(\cdot)v,\, F_\gamma^\ast (\cdot)v\in D([0,T];U)
     \quad\text{for all }v\in V   \text{ and } \gamma>0;\\
 {\rm (ii)}&\, \sup_{\gamma>0}\sup_{k\in\N}\norm{\phi_\gamma^k}_\infty<\infty; \label{eq.M-weak-diagonal-conv-bound}\\
    {\rm (iii)}&\,
    \sup_{\gamma >0}\sup_{k\in\N}\norm{\phi_\gamma^k}_{TV_2}<\infty; \label{eq.M-weak-diagonal-conv-TV}\\
  \begin{split}
    {\rm (iv)}&\, \text{for each $n\in\N$ we have}\\
    &\, \lim_{\gamma\to\infty}\big(\phi_\gamma^1, \dots, \phi_\gamma^n\big)
           = \big(\phi^1, \dots, \phi^n\big) \quad \text{in $\big(D([0,T];\R^n),d_M\big)$},
    \end{split}
           \label{eq.M-weak-diagonal-conv-phiinM}
\end{align}
then it follows that
\begin{align*}
\lim_{\gamma\to\infty} \big(F_\gamma \ast L(t)\colon t\in [0,T]\big)
= \big(F \ast L(t)\colon t\in [0,T]\big)
\end{align*}
weakly in probability in $D([0,T];V)$.
\end{corollary}
\begin{proof}
The proof is analogous to the proof of Theorem \ref{th.Dconv} for $D=V$ but only the estimate \eqref{eq.pXestimate} is derived in the following way: for each $v\in V$ and $\gamma>0$  Lemma
\ref{le.MarcusRosHilbert-realvalued} and Cauchy--Schwarz inequality imply
\begin{align*}
&E\left[\sup_{t\in [0,T]} \abs{\int_0^t F_\gamma^\ast(t-s)v\,dX(s) } \right]\\
&\qquad\qquad\le  \varkappa \sqrt{2T}\sum_{k=1}^\infty \abs{\scapro{v}{e_k}}\norm{\phi_\gamma^k(\cdot)}_{TV_2}\left(\int_{\norm{u}\le \alpha} \abs{\scapro{u}{Ge_k}}^2\,\nu(du)\right)^{1/2}\\
&\qquad\qquad\le
  \varkappa \sqrt{2T}\sup_{k\in\N}\norm{\phi_\gamma^k(\cdot)}_{TV_2} \left(\sum_{k=1}^\infty \scapro{v}{e_k}^2\right)^{1/2}\left(\sum_{k=1}^\infty \int_{\norm{u}\le \alpha} \scapro{u}{Ge_k}^2\,\nu(du)\right)^{1/2}\\
&\qquad\qquad =
  \varkappa \sqrt{2T}\sup_{k\in\N}\norm{\phi_\gamma^k(\cdot)}_{TV_2}\norm{v} \left( \int_{\norm{u}\le\alpha} \norm{G^\ast u}^2\, \nu(du)\right)^{1/2},
\end{align*}
where $\varkappa:=32\sqrt{2} \int_0^1 \sqrt{\ln(1/s)}\,ds$ and $\alpha>0$ denotes the bound of the truncation function.  Condition \eqref{eq.M-weak-diagonal-conv-TV} enables us to choose for every $\epsilon_1$, $\epsilon_2>0$ the constant  $\alpha>0$ small enough such that
\begin{align*}
 P\left(\sup_{t\in [0,T]} \abs{\int_0^t F_\gamma^\ast(t-s)v\,dX(s)}\ge \epsilon_1\right)
 \le \frac{\epsilon_2}{2} \qquad\text{for all }\gamma>0,
\end{align*}
which corresponds to inequality \eqref{eq.pXestimate}. Now we can follow the proof of Theorem \ref{th.Dconv}
once the Conditions \eqref{eq.proMconv-bound} and \eqref{eq.proMconv-conv}
are established. Cauchy--Schwarz inequality implies  for each $u\in U$ and $v\in V$
\begin{align*}
  \sup_{\gamma>0} \norm{\scapro{F_\gamma(\cdot)u}{v}}_\infty^2
  &\le \norm{G^\ast u}^2 \norm{v}^2\sup_{\gamma>0}\sup_{k\in\N} \norm{\phi_\gamma^k}^2_\infty
  <\infty,
\end{align*}
which establishes Condition \eqref{eq.proMconv-bound}. For the last part, fix $u\in U$ and $v\in V$ and define for each $\gamma>0$ and $n\in \N$  the functions
\begin{align*}
  f&:=\scapro{F(\cdot)u}{v}, \qquad f^\gamma:=\scapro{F_\gamma(\cdot)u}{v},\\
  f_n&:=\sum_{k=1}^n\scapro{u}{Ge_k}\scapro{v}{e_k}\phi^k ,
  \qquad f_n^\gamma:=\sum_{k=1}^n\scapro{u}{Ge_k}\scapro{v}{e_k}\phi_\gamma^k .
\end{align*}
Cauchy--Schwarz inequality implies
\begin{align*}
 \sup_{\gamma>0} \norm{f_n^\gamma-f^\gamma}_\infty^2
  &=\sup_{\gamma>0} \norm{\sum_{k=n+1}^\infty \phi_\gamma^k(\cdot) \scapro{u}{Ge_k}\scapro{v}{e_k}}^2_\infty \\
  &\le \sup_{\gamma>0}\sup_{k\in\N}\norm{\phi_\gamma^k}_\infty^2 \left(\sum_{k=n+1}^\infty \scapro{G^\ast u}{e_k}^2\right)
  \left(\sum_{k=n+1}^\infty \scapro{v}{e_k}^2\right)
  \to 0, \quad n\to\infty,
\end{align*}
and analogously we obtain $\norm{f_n-f}_{\infty}\to 0$ as $n\to\infty$. Since Condition
 \eqref{eq.M-weak-diagonal-conv-phiinM} guarantees by Theorem 12.7.2 in \cite{whitt02} that
$ f_n^\gamma \to f_n$ as $\gamma\to\infty$ in  $\big(D([0,T];\R),d_M)$ for all $n\in\N$,
Lemma \ref{le.sup+M} implies
\begin{align*}
  \lim_{\gamma\to\infty} \scapro{F_\gamma(\cdot)u}{v}=\scapro{F(\cdot)u}{v}
   \qquad\text{in $\big(D([0,T];\R),d_M)$,}
\end{align*}
which shows Condition \eqref{eq.proMconv-conv} and completes the proof.
\end{proof}

\section{Integrated Ornstein-Uhlenbeck process}\label{se.O-U-process}

Let $V$ be a separable Hilbert space and let $A$ be the generator of a strongly continuous semigroup 
$(S(t))_{t\ge 0}$ in $V$. For $\gamma>0$ consider the equation
\begin{align}\label{eq.eCauchy}
\begin{split}
  dY_\gamma(t)&=\gamma AY_\gamma (t)\,dt + G\,dL(t)\qquad\text{for }t\in [0,T],\\
   Y_{\gamma}(0)&=0,
\end{split}
\end{align}
where $L$ denotes a L{\'e}vy process in $U$ and $G\in \L(U,V)$. A progressively measurable stochastic process 
$(Y_\gamma(t)\colon t\in [0,T])$ is called a {\em weak solution of \eqref{eq.eCauchy}} if it satisfies 
for every $v\in \D(A^\ast)$ and $t\in [0,T]$ $P$-a.s. the equation
\begin{align*}
  \scapro{Y_\gamma(t)}{v}=\int_0^t \scapro{Y_\gamma(s)}{A^\ast v}\, ds + \scapro{L(t)}{G^\ast v}.
\end{align*}

Since $A_\gamma:=\gamma A\colon \D(A)\subseteq V\to V$ is the generator of the $C_0$-semigroup $(S_\gamma(t))_{t\ge 0}$ 
where $S_\gamma(t):=S(\gamma t)$ for all $t\ge 0$, Theorem 2.3 in \cite{Chojnowska87} 
implies that there exists a unique weak solution $ Y_\gamma:=(Y_\gamma(t)\colon  t\in [0,T])$ of 
\eqref{eq.eCauchy}, and which can be represented by
\begin{align*}
  Y_{\gamma}(t)=\int_0^t S_{\gamma}(t-s)G\, dL(s)\qquad\text{for all }t\in [0,T].
\end{align*}
We require that $Y_\gamma$ has c{\`a}dl{\`a}g trajectories, which is satisfied for example if 
the semigroup $(S(t))_{t\ge 0}$ is analytic or contractive; see \cite{PesZab07}. 
Then the integrated Ornstein-Uhlenbeck process $(X_\gamma(t)\colon  t\in [0,T])$ is defined by
\begin{align*}
  X_\gamma(t):= \gamma \int_0^t Y_\gamma (s)\,ds \qquad\text{ for all } t\in [0,T].
\end{align*}

\begin{corollary}
Assume that the semigroup $(S(t))_{t\ge 0}$ is diagonalisable, i.e.\ there exists an orthonormal 
basis $e:=(e_k)_{k\in\N}$ of $V$ such that
\begin{align*}
   S(t)e_k=e^{-\lambda_k t}e_k
   \qquad \text{for all }t\ge 0, \, k\in\N,
\end{align*}
for some $\lambda_k> 0$ with $\lambda_k\to \infty $ for $k\to\infty$.
Then the integrated Ornstein--Uhlenbeck process $X_\gamma$ satisfies
\begin{align*}
  \lim_{\gamma\to\infty}\big(AX_\gamma(t)\colon  t\in [0,T]\big)= \big(-GL(t)\colon  t\in [0,T]\big)
\end{align*}
in probability in the product topology $\big(D([0,T];V), d_M^{{\,e}}\big)$.

\end{corollary}
\begin{proof}
For every $t\in [0,T]$ and $v\in \D(A^\ast)$ Fubini's theorem implies
by \eqref{eq.integralswap}
\begin{align}\label{eq.eCauchyintFubini}
  \scapro{AX_\gamma(t)}{v}
  &= \gamma \int_0^t \scapro{Y_\gamma(s)}{A^\ast v}\,ds\notag\\
  &=\gamma \int_0^t \left(\int_0^s G^\ast S_\gamma^\ast (s-r)A^\ast v\,dL(r)\right)\,ds\notag\\
  &=\gamma \int_0^t\left(  \int_r^t G^\ast S_\gamma^\ast (s-r)A^\ast v\,ds\right)\,dL(r).
\end{align}
It follows from properties of the adjoint semigroup (see \cite[Proposition 1.2.2]{vanNeervenAdj})
that for every $r\in [0,t]$ we have
\begin{align}\label{eq.intadsemi}
 \int_r^t G^\ast S_\gamma^\ast (s-r)A^\ast  v\,ds
  = \frac{1}{\gamma} G^\ast A^\ast_\gamma \int_0^{t-r} S_\gamma^\ast (s) v\,ds
  =\frac{1}{\gamma}  \left( G^\ast S^\ast_\gamma(t-r)v - G^\ast v\right).
\end{align}
Define a L{\'e}vy process $K$ in $V$ by $K(t):=GL(t)$ for all $t\ge 0$. 
By combining \eqref{eq.intadsemi} with \eqref{eq.eCauchyintFubini} we obtain from \eqref{eq.integralswap}
and \eqref{eq.integral-imageLevy} that
\begin{align*}
  \scapro{AX_\gamma(t)}{v}
  =\int_0^t \left( G^\ast S^\ast_\gamma(t-r)v - G^\ast v\right)\,dL(r)
  =\scapro{\int_0^t \big(S_\gamma(t-r)- \Id\big)\,dK(r)}{v}.
\end{align*}
Consider the functions
\begin{align*}
  &F\colon [-1,T]\to \L(V,V),\qquad F(t)=
  \begin{cases}
    -\Id , &\text{if }t\in [0,T],\\
     0 , &\text{if }t\in [-1,0),
  \end{cases}\\
  &F_\gamma\colon [-1,T]\to \L(V,V), \qquad F_\gamma (t)=
  \begin{cases}
   S(\gamma\,t)-\Id, &\text{if } t\in [0,T],\\
   0, &\text{if }t\in [-1,0).
  \end{cases}
\end{align*}
Defining the functions above on $[-1,T]$ and not only on $[0,T]$ with a jump at $0$ enables us to
consider c{\`a}dl{\`a}g functions. By means of Corollary \ref{co.Mproductconv-diagonal} 
(with an obvious adaption for considering the interval $[-1,T]$) we show that
\begin{align}\label{eq.AXtoL}
  \lim_{\gamma\to \infty}
  \left(\int_0^t  \big(S_{\gamma}(t-r)- \Id\big)\,dK(r)\colon  t\in [-1,T]\right)
  =  \left(\int_0^t F(t-s)\,dK(s)\colon  t\in [-1,T]\right)
\end{align}
in probability in the product topology $\big(D([-1,T];V), d_M^{{\,e}}\big)$.
The functions $F_\gamma$ and $F$ are of the form
\begin{align*}
  F^\ast (t)e_k &=\phi^k(t)\, e_k,\qquad\qquad  F_\gamma^\ast (t)e_k =\phi_\gamma^k(t) e_k \qquad\text{for all }t\in [-1,T], \,k\in\N,
\intertext{where the real-valued functions $\phi^k$, $\phi^k_\gamma\colon [-1,T]\to\R$ are defined by}
\phi^k(t)&=-\1_{[0,T]}(t), \qquad \quad \phi_\gamma^k(t)=\1_{[0,T]}(t)(e^{-\lambda_k \gamma t}-1).
\end{align*}
 For every $k\in\N$  the sequence $(\phi_\gamma^k)_{\gamma>0}$ meets Condition \eqref{eq.thmMconvprod-diagonal-bound}, as
\begin{align*}
  \sup_{\gamma>0}\norm{\phi_\gamma^k}_\infty
  =\sup_{\gamma>0}\sup_{s\in [0,T]} \abs{e^{-\lambda_k \gamma s}- 1}\le1.
\end{align*}
Since $\phi_\gamma^k$ is a decreasing function it has finite variation and
thus $\norm{\phi_\gamma^k}_{TV_2}=0$
which verifies Condition \eqref{eq.thmMconvprod-diagonal-TV}.
Since $\phi_\gamma^k$ is monotone for each $\gamma>0$, $k\in\N$ and satisfies
for each $k\in\N$
\begin{align*}
  \lim_{\gamma\to\infty}\phi_\gamma^k(s)= \phi^k(s)
  \qquad\text{for all }s\in [-1,T]\setminus\{0\},
\end{align*}
it follows from Corollary 12.5.1 in \cite{whitt02}, that $\phi_\gamma^k\to \phi^k
$ as $\gamma\to \infty$ in $\big(D([0,T];\R), d_M\big)$,
which is Condition \eqref{eq.thmMconvprod-diagonal-phiinM}. Thus,
we can apply Corollary  \ref{co.Mproductconv-diagonal} to conclude \eqref{eq.AXtoL}.
\end{proof}


\end{document}